\documentclass[smallextended]{svjour3} 
\usepackage[sort, numbers]{natbib}

\usepackage{amsfonts,amsmath,graphicx,fullpage,bbm,amsthm}
\usepackage{amssymb,multirow,verbatim}
\usepackage{acronym,wrapfig,plain,mathrsfs,enumerate,relsize,color}
\newtheorem{assumption}{Assumption}
\newtheorem{algorithm}{Algorithm} 

\usepackage{algorithm}
\usepackage{subfig}
\usepackage{algorithmic}
\usepackage{boxedminipage}
\usepackage[justification=centering]{caption}
\def\bko{{\rm 1\kern-.17em l}}
\usepackage{wrapfig}
\usepackage{lipsum}

\newcommand{\an}[1]{{\color{black}#1}}
\newcommand{\fy}[1]{{\color{black}#1}}
\newcommand{\us}[1]{{\color{black}#1}}

\newcommand{\fyR}[1]{{\color{black}#1}}

\renewcommand{\theenumi}{\arabic{enumi} )}

\def\be{\begin{enumerate}}
\def\ee{\end{enumerate}}

\def\argmin{\mathop{\rm argmin}}

 \newcommand{\remove}[1]{}
\newcommand{\EXP}[1]{\mathsf{E}\!\left[#1\right] }

\def\sF{\mathcal{F}}

\def\Real{\mathbb{R}}
\def\g{\gamma}

\def\e{\epsilon}
\def\a{\alpha}

\def\argmin{\mathop{\rm argmin}}

\title{On stochastic mirror-prox algorithms for stochastic Cartesian variational inequalities: randomized block coordinate and optimal averaging schemes} 

\author{Farzad~Yousefian,   
        Angelia~Nedi\'c, and   
		Uday V.~Shanbhag} \institute{Yousefian is in the School of Indust. Engg. \& Mgnt., Oklahoma State Univ., Stillwater, OK 74074, USA, Nedi\'c is in the School of Elec., Comp., and Energy Engg., Arizona State Univ., Tempe, AZ 85287, USA,
		while Shanbhag is with the Department of Indust. and
			Manuf. Engg., Penn. State Univ.,
						  University Park, PA 16802, USA. They are
							  reachable at 
\tt\small{farzad.yousefian@okstate.edu}, \tt\small{angelia.nedich@asu.edu}, and \tt\small{
	udaybag@psu.edu}.}

\begin{document}
\maketitle

\begin{abstract}
Motivated by multi-user optimization problems and non-cooperative Nash games in uncertain regimes, we consider stochastic Cartesian variational inequality problems (SCVI) where the set is given as the Cartesian product of a collection of component sets. First, we consider the case where the number of the component sets is large. For solving this type of problems the classical stochastic approximation methods and their prox generalizations are computationally inefficient as each iteration becomes costly. To address this challenge, we develop a randomized block stochastic mirror-prox (B-SMP) algorithm, where at each iteration only a randomly selected block coordinate of the solution vector is updated through implementing two consecutive projection steps. Under standard assumptions on the problem and settings of the algorithm, we show that when the mapping is strictly pseudo-monotone, the algorithm generates a sequence of iterates that converges to the solution of the problem \fyR{almost surely.} To derive rate statements, we assume that the maps are strongly \fyR{pseudo-monotone} and obtain \fyR{a non-asymptotic rate of $\mathcal{O}\left(\frac{d}{k}\right)$ in mean-squared error}, where $k$ is the iteration number and $d$ is the number of \fyR{component sets}. Second, we consider large-scale stochastic optimization problems with convex objectives. For this class of problems, we develop a new averaging scheme for the B-SMP algorithm. Unlike the classical averaging stochastic mirror-prox (SMP) method where a decreasing set of weights for the averaging sequence is used we consider a different set of weights that are characterized in terms of the stepsizes. We show that by using \fyR{such} weights, the objective \fyR{values} of the averaged sequence converges to \fyR{the} optimal value in the mean sense at the rate \fyR{$\mathcal{O}\left(\frac{\sqrt{d}}{\sqrt{k}}\right)$}. Both of the rate results \fyR{appear} to be new in the context of SMP algorithms. Third, we consider \fyR{SCVIs and develop an SMP algorithm} that employs the new weighted averaging scheme. We show that the expected value of a suitably defined gap function converges to zero at the optimal rate \fyR{$\mathcal{O}\left(\frac{1}{\sqrt{k}}\right)$}, extending the previous rate results of the SMP algorithm.
\end{abstract}


\section{Introduction}\label{sec:introduction}
Variational inequality (VI) problems, first introduced in 1960s, provide a unifying framework for capturing a wide range of applications arising in operations research, finance, and economics (cf. \cite{facchinei02finite, Rockafellar98, LanAccVI, wang2015, IusemIncrementalVI}). Given a set $X \subset \Real^n$ and a mapping $F:X \rightarrow \mathbb{R}^n$, a variational inequality problem is denoted by VI$(X,F)$, where the goal is to find a vector $x^* \in X$ such that \begin{align}\label{def:VI}
\langle F(x^*),x-x^*\rangle\geq 0\qquad \hbox{for all }x \in X.\tag{VI}
\end{align} 
In the presence of uncertainty associated with the mapping $F$ or the set $X$, stochastic generalizations of variational inequalities have been developed and \fyR{studied} \cite{Wets12SVI, Xu10,Nem11}. Such situations occur \fyR{in two cases: (i) when }the mapping $F$ is characterized via an expected value of a stochastic mapping. \fyR{In this case, } when the underlying probability distribution is unknown \fyR{or} the number of random variables is large, direct application of deterministic methods becomes challenging; \fyR{(ii) The same applies when the set $X$} is characterized by uncertainty. While the the former case finds relevance in stochastic optimization and stochastic Nash games, the latter occurs in areas such as traffic equilibrium problems with uncertain link capacities (cf. Sec. 4 in \cite{Wets12SVI}). 

Motivated by multi-user stochastic optimization problems and non-cooperative Nash games, in this paper, we are interested in solving  stochastic Cartesian variational
inequality problems. Consider (deterministic) sets $X_i \in \Real^{n_i}$ for $i=1,\ldots,d$. Let the random
vector $\xi$ be defined as $\xi:\Omega
\rightarrow \Real^m$ and {$(\Omega,{\cal F}, \mathbb{P})$ denote the
associated probability space. A stochastic Cartesian variational
inequality (SCVI) is a problem of the form \eqref{def:VI} such that the set $X$ and mapping $F$ are defined as follows:
\begin{align}\label{def:XFofSCVI}
X\triangleq\prod_{i=1}^d X_i, \quad F(x)\triangleq\EXP{F(x,\xi)},
\end{align}
where $F(x,\xi): X\times \Real^m \to \Real^n$ denotes the random mapping, the mathematical expectation is taken with respect to the random vector $\xi$, and $n\triangleq\sum_{i=1}^d n_i$.
Here the set $X$ is a Cartesian product of the component sets $X_i$.  
This problem can be represented as determining a vector $x^*=({x^*}^1; {x^*}^2;\ldots; {x^*}^d)\in X$ such that for all $i=1,\ldots,d$,
\begin{align}\label{eqn:SCVI}
\langle\EXP{F_i(x^*,\xi)},x^i-{x^*}^i\rangle \geq 0 \qquad \hbox{ for all $x^i \in X_i$} \tag{SCVI}.
\end{align} 
where $x^i$ and the component mapping $F_i(x,\xi)$ are such that \[x=(x^1;x^2;\ldots;x^d), \quad\hbox{and} \quad F(x,\xi)=(F_1(x,\xi);\ldots;F_d(x,\xi)).\]
Throughout, we assume that 
the expected mappings $\EXP{F_i(x,\xi)}:\Real^n\to\Real^{n_i}$ are well-defined (i.e., the expectations are finite). Our work is motivated by the following two classes of problems that can be both represented \fyR{by} \eqref{eqn:SCVI}:

\noindent \fyR{\textbf{Stochastic non-cooperative Nash games:}}
Consider a classical non-cooperative Nash game among $d$ players (agents). Each player is associated with a strategy set and a cost function. Let $x^i$ denote the strategy (decision variable) of the $i$th player, where $x_i$ belongs to the set of all possible actions of the $i$th player denoted by $X_i \subset \Real^{n_i}$. Let $f_i((x^i;x^{-i}),\xi)$ denote the random cost function of the $i$th player that is in terms of action of the player $x^i$, actions of other players denoted by $x^{-i}$, and a random variable $\xi$ representing the state of the game. The goal of each player is to minimize the expected value of the cost function for any arbitrary strategies of the other players, i.e., $x^{-i}$, by solving the following problem:
\fyR{\begin{align*}
\displaystyle \mbox{minimize} & \qquad  {\EXP{f_i((x^i;x^{-i}),\xi)}} \\
\mbox{subject to} & \qquad x^i \in X_i.
\end{align*}}
A Nash equilibrium is a tuple of strategies $x^*=({x^*}^1;{x^*}^2;\ldots;{x^*}^d)$ where no player can obtain a lower cost by deviating from his strategy if the strategies of the other players remain unchanged. 
Under the validity of the interchange between the expectation and the
derivative operator, the resulting equilibrium conditions of this stochastic Nash
game are compactly captured \fyR{by VI$(X,F)$}
where 
$ X \triangleq \prod_{i=1}^d X_i$ and {$F(x) = (F_1(x);\ldots; F_d(x))$}
with $F_i(x)=\EXP{\nabla_{x^i} f_i(x,\xi)}$.
Such problems arise in communication
networks~\cite{alpcan02game,alpcan03distributed,yin09nash2}, 
competitive interactions in cognitive radio
networks~\cite{aldo1,scutari10monotone,koshal11single2}, and in power markets~\cite{KShKim11,KShKim12,SIGlynn11}. 

\noindent \fyR{\textbf{Block structured stochastic optimization:}} Motivated by multi-agent
decision-making problems such as rate allocation problems in communication
networks~\cite{Kelly98,Srikant04,ShakSrikant07}, we consider the following block structured stochastic optimization problem:
\begin{align}\label{prob:SOP}
\displaystyle \mbox{minimize} & \qquad f(x)\triangleq\EXP{f(x,\xi)}\tag{SCOP}\\
\mbox{subject to} & \qquad x \in X\triangleq\prod_{i=1}^d X_i,\notag
\end{align}
where $\xi \in \Real^m$ is a random variable associated with a probability distribution, function $f(\cdot,\xi):X\to \Real$ is continuous for all $\xi$, and the set $X \in \Real^n$ is the Cartesian product of the sets $X_i \in \Real^{n_i}$, \fyR{with} $n\triangleq\sum_{i=1}^d n_i$. \fyR{Note} that the optimality conditions of  \eqref{prob:SOP} can be represented as \eqref{eqn:SCVI}, where $F_i(x)=\EXP{\nabla_{x^i} f(x,\xi)}$.

Our primary interest in this paper lies in solving \eqref{eqn:SCVI} when the number of component sets $X_i$, i.e., $d$, is very large. Computing the solution to this class of problems is challenging mainly due to presence of \fyR{\textit{uncertainty} and \textit{high dimensionality}} of the solution space. In what follows, we review some of the existing methods in addressing these challenges:

\noindent \textit{Addressing uncertainty in optimization and VI regimes:} Contending with uncertainty in solving variational inequalities has been carried out through the application of Monte-Carlo sampling schemes. Of these, sample average approximation (SAA) scheme proposes a framework in that the expected value of the stochastic mapping is approximated via the average over a large number of samples (cf. \cite{shap03sampling}, Chapter 6). However, it has been discussed that the SAA approach is computationally inefficient when the sample size is large \cite{nemirovski_robust_2009}. A counterpart to SAA schemes is the stochastic approximation (SA) methods and their generalizations where at each iteration, a sample (or a small batch) of the stochastic mapping is used to update the solution iterate. It was first in 1950s when Robbins and Monro \cite{robbins51sa} developed the SA method to address stochastic root-finding problems. Due to their computational efficiency in addressing problems with a large number of samples and also their adaption to on-line settings,  SA methods have been very successful in solving optimization and equilibrium problems with uncertainties. \fyR{Jiang and Xu} \cite{xu2008} appear amongst the first who applied SA methods to solve stochastic variational inequalities with smooth and strongly monotone mappings. Extension of that work was studied by Koshal et al. \cite{koshal12regularized} addressing merely monotone stochastic VIs. More recently, we developed a regularized smoothing SA method to address stochastic VIs with non-Lipschitzian and merely monotone mappings \cite{Farzad3}. In recent years, prox generalization of SA methods were developed \cite{Nem04,nemirovski_robust_2009,Nem11,Lan-VI-13} for solving smooth and nonsmooth stochastic convex optimization problems and variational inequalities. In \cite{nemirovski_robust_2009}, a stochastic mirror descent (SMD) method is proposed to solve stochastic optimization problems with convex objectives. SMD method generalizes the SA method in that the Bregman distance function is employed in vector spaces equipped with non-Euclidean norms. The convergence properties and rate analysis of this class of solution methods rely on the monotonicity of the gradient mapping. To address variational inequalities requiring weaker assumptions, i.e., pseudo-monotone mappings, Korpelevich \cite{korp76} developed the extragradient method in 1970s. Since the extragradient method requires two projections per iteration, the computational complexity is twice than that of its classical gradient counterpart. However, it benefits significantly from the extra step by addressing VIs with weaker assumptions, i.e., VIs with pseudo-monotone mappings. Dang et al. \cite{Lan-VI-13} developed non-Euclidean extragradient methods addressing generalized monotone VIs and derived the convergence rate statements under smoothness properties of the problem and the distance generator function. In \cite{Nem11}, Juditsky et al. developed a stochastic mirror-prox (SMP) method to solve stochastic VIs with monotone operators. Loosely speaking, the SMP method is the prox generalization of extragradient scheme to stochastic settings. It is shown that under an averaging scheme, the SMP method generates iterates that converge to a weak solution of the stochastic VI. Kannan and Shanbhag (see \cite{Aswin17arxiv,Aswin-ACC14}) studied almost sure convergence of extragradient algorithms in solving stochastic VIs with pseudo-monotone mappings and derived optimal rate statements under a strong pseudo-monotone condition. \fyR{ Recently, Iusem et al. \cite{IusemExtra2017} developed an extragradient method with variance reduction for solving stochastic variational inequalities requiring only pseudo-monotonicity.} Motivated by the recent developments in extragradient methods and their generalizations, in this paper, we consider SMP methods.

\noindent \textit{Addressing high dimensionality in optimization and VI regime:} When the dimensionality of the solution space is huge, e.g., $n=\sum_{i=1}^dn_i$ exceeds $10^{12}$, the direct implementation of the aforementioned solution methods becomes problematic. Specifically, SMD and SMP algorithms both require performing arithmetic operations of order $n$ per iteration. Moreover, the projection step (i.e., minimizing the prox function) in both methods is a another source of inefficiency for huge size problems. For problems with Cartesian solution spaces, the computational effort for projection can be reduced by decomposing the projection into $d$ projections corresponding to the set components $X_i$ at each iteration \cite{Farzad2}. However, this approach still requires updating and storing $d$ component solution iterates of sizes $n_i$ for all $i=1,\ldots,d$ per iteration. To address this issue and improve the efficiency of the underlying solution method, coordinate descent (CD) and more generally block coordinate descent (BCD) methods have been developed and studied in recent decades. Ortega and Rheinboldt \cite{ortega70} studied the concept of such approaches as ``univariate relaxation". The convergence properties of the CD methods was studied in 80s and 90s by researchers such as Tseng \cite{TsengCD1,TsengCD2,TsengCD3}, Bertsekas, and Tsitsiklis \cite{BertTsitBookPar}, and more recently in \cite{Wotao13,Richtarik14, Richtarik15} (see \cite{wrightCD15} for \fyR{detailed} review on CD methods). Nesterov \cite{nest10} appears amongst the first \fyR{to provide} a comprehensive iteration complexity analysis for randomized BCD methods for \fyR{solving optimization problems with smooth and convex objective functions.} In a recent work, Dang and Lan \cite{LanSBMD} developed a stochastic block coordinate mirror descent algorithm (SBMD) to address smooth and nonsmooth stochastic optimization problems with Cartesian constraint sets. in each iteration of SBMD scheme, only a randomly selected block of \fyR{iterates} is updated to improve the iteration complexity substantially. 

\textit{Summary and main contributions}:  \fyR{We develop two variants of the SMP algorithm.} In the first part of the paper, motivated by recent advancements in BCD methods, we develop a randomized block stochastic mirror-prox (B-SMP) algorithm to address SCVIs when the number of component sets is huge. At each iteration of the B-SMP scheme, first a block is selected randomly. Then, the selected block of the solution iterate is updated through performing two successive projection steps on the corresponding component set. We provide the convergence and rate analysis of the B-SMP algorithm under both non-averaging and averaging schemes as will be explained in details in the following discussion. In the second part of the paper, to address SCVIs with small or medium number of component sets, we develop a SMP method in that at each iteration, the solution iterate is updated at all the blocks through two projections on each component set. For this class of algorithms, we employ a new averaging scheme and derive an optimal rate statement. In what follows, we summarize the main contributions of our work:
 
\noindent \textit{(i) Addressing large-scale VIs}: While both SMD and SMP algorithms have been employed in the literature to address VIs \cite{Nem04, Nem11, Farzad3}, our algorithm appears to be the first that is capable of computing the solution to a large-scale Cartesian variational inequality in both deterministic and stochastic regimes. To this end, in contrast with the SBMD algorithm in \cite{LanSBMD}, we employ a randomized block scheme for the mirror-prox method. It is worth noting that the analysis of the B-SMP method \fyR{cannot be directly extended} mainly because the same randomly selected variable for coordinates is used in both projection steps, resulting a dependency between the uncertainty involved in the two projection steps. \\
\noindent \textit{(ii) Convergence and rate analysis}:  In contrast with earlier work on stochastic mirror-prox methods \cite{Nem11} where the convergence and rate analysis is performed under monotonicity assumption for stochastic VIs using averaging schemes, we first consider a non-averaging random variant of the mirror-prox method. We study the properties of the iterate generated by the B-SMP algorithm under pseudo-monotonicity assumption (see Lemma \ref{lemma:pseudo}). Next, under strict pseudo-monotonicity assumption, we prove convergence of the generated iterate to the solution of \eqref{eqn:SCVI} in an almost sure sense, extending the results of \cite{Aswin-ACC14} to the block coordinate settings. When considering SCVIs with strongly pseudo-monotone mappings, we obtain a bound of the order $\mathcal{O}\left(\frac{d}{k}\right)$ on the mean squared error, where $k$ is the iteration number and $d$ is the number of blocks extending results in \cite{Aswin-ACC14} to the block coordinate regime. This result differs from the rate analysis of \cite{LanSBMD} for the SBMD algorithm for stochastic optimization problems with strongly convex objectives in three aspects: (a) While the SBMD method addresses the optimization regime, our rate result applies to the broader class of problems, i.e., SVCIs; (b) The assumption of strong pseudo-monotonicity in our work is weaker than the strong monotonicity of the gradient mapping in \cite{LanSBMD}; (c) \fyR{In contrast with the SBMD scheme where an averaging scheme with a constant stepsize rule is employed for addressing problem (SCOP) {(cf. Corollary 2.2. in \cite{LanSBMD})}, here we use a non-averaging randomized block coordinate scheme for problem (SCVI). Note that averaging schemes with constant steplength necessarily produce sequences whose limit points are approximate solutions, while our scheme produces sequences that are asymptotically a.s. convergent. While in \cite{LanSBMD} the convergence rate is derived both under convexity and strong convexity assumptions of the objective function, here we are able to derive the convergence rate under a weaker assumption of strong pseudo-monotonicity of the mapping (see Proposition \ref{prop:ratestrongpse}).} \\
\noindent \textit{(iii) Developing optimal weighted averaging schemes}: In the second  part of our analysis, to derive rate statements for the B-SMP method under monotonicity assumption, we restrict our attention to the class of stochastic convex optimization problems with Cartesian product of feasible sets. We develop a new averaging scheme for the B-SMP algorithm, \fyR{which updates a weighted average of the form:}
\begin{align*}
\bar x_k\triangleq\sum_{t=0}^k\alpha_t x_t, \qquad \hbox{for all } k\geq 0, 
\end{align*}
where $x_0,\ldots, x_t$ are the generated iterates of the B-SMP method and $\alpha_t$ is the weight assigned to the iterate $x_t$. In the literature, averaging schemes have been employed in addressing stochastic optimization problems and SVIs to derive error bounds. In the 90s, Polyak and Juditsky \cite{Polyak92} employed averaging schemes in stochastic approximation schemes and studied their convergence properties for a variety of problems. Later, Nemirovski et al. developed averaging schemes for the stochastic mirror descent method \cite{nemirovski_robust_2009} and stochastic mirror-prox algorithm \cite{Nem11} and derived the convergence \fyR{rate $\mathcal{O}\left(\frac{\ln(k)}{\sqrt{k}}\right)$} when $\alpha_k$ is chosen to be $\frac{\g_t}{\sum_{t=0}^k\g_t}$ and $\g_t$ is the stepsize sequence. They also showed that the optimal convergence \fyR{rate $\mathcal{O}\left(\frac{1}{\sqrt{k}}\right)$} can be achieved if a window-based averaging scheme is employed. More recently, Nedi\'c and Lee \cite{Nedic14} developed a new averaging scheme for SMD algorithm in that $\alpha_t$ is given as 
$\frac{\g_t^{-1}}{\sum_{t=0}^k\g_t^{-1}}$. They showed that under this different set of weights, the SMD algorithm admits the optimal rate of convergence without requiring a window-based averaging scheme. Motivated by this contribution and our previous work on SA method \cite{Farzad3} and stochastic extragradient scheme \cite{Farzad-CDC14}, in this work, we extend this result by considering $\alpha_t$ as $\frac{\g_t^{r}}{\sum_{t=0}^k\g_t^{r}}$ where $r$ is an arbitrary scalar. We show that for any arbitrary fixed $r<1$, the B-SMP algorithm generates iterates such that the objective function of $\bar x_k$ converges to its optimal value in mean admitting the convergence \fyR{rate $\mathcal{O}\left(\frac{\sqrt{d}}{\sqrt{k}}\right)$}. We also analyze the constant factor in terms of the problem and algorithm parameters and the parameter $r$. \fyR{We note that complexity analysis for averaging schemes in combination with constant stepsizes has been done in \cite{LanSBMD}. Unlike, the work in \cite{LanSBMD}, our averaging schemes are employed with a diminishing stepsize.}

In the last part of the paper, we consider SCVIs \fyR{with monotone mappings.} We employ the new averaging scheme in the classical SMP algorithm and show that the optimal convergence rate of the order $\mathcal{O}\left(\frac{1}{\sqrt{k}}\right)$ can be achieved without requiring a window-based averaging, extending the results of \cite{Nem11}.

The rest of the paper is organized as follows. In Section \ref{sec:pre}, we \fyR{discuss} the prox functions and some of their main properties. We also cite a few results that will be used later in the analysis of the paper. Section \ref{sec:alg} contains the outline of the B-SMP algorithm and the main assumptions. In Section \ref{sec:conv}, we present the convergence theory for the developed algorithm in an almost sure sense, and present the rate analysis of the proposed scheme for SCVIs and SCOPs in Section \ref{sec:rate}. In Section \ref{sec:SMP}, we present the SMP algorithm for SCVIs with optimal averaging schemes. Lastly, we \fyR{provide some} concluding remarks in Section \ref{sec:conRem}.

\fyR{\textbf{Notation:}}\label{sec:not}
For any vector $x \in \Real^n$, we let $x^i \in \Real^{n_i}$ denote the $i$th block coordinate of $x$ such that $x=(x^1;x^2;\ldots,x^d)$. We use subscript $i$ to denote the $i$th block of a mapping in $\Real^n$, e.g., for $F:\Real^n \to \Real^n$, we have $F(x)=(F_1(x);F_2(x);\ldots;F_d(x))$. For any $i = 1,\ldots,d$ and $x\in \Real^n$, we let $x^{-i}$ denote the collection of blocks $x^j$ for any $j \neq i$, such that $x=(x^i;x^{-i})$. For any $i=1,\ldots,d$, $\|\cdot\|_i$ denotes the general norm operator on $\Real^{n_i}$ and its dual norm is defined by $\|x\|_{*i}=\sup_{\|y\|\leq 1}\langle y,x\rangle$ for $x,y \in \Real^{n_i}$. For any $u,v \in \Real^{n_i}$, $\langle u,v\rangle$ denotes the inner product of vectors $u$ and $v$ in $\Real^{n_i}$. When $u,v \in \Real^n$, we let $\langle u,v\rangle\triangleq\sum_{i=1}^d \langle u^i,v^i\rangle$. We define norm $\|\cdot\|$ as $\|x\|^2\triangleq\sum_{i=1}^d\|x^i\|_i^2$ for any $x \in \Real^n$, and denote its dual norm by$\|\cdot\|_*$. We write $\textbf{I}_n$ to denote the identity matrix in $\Real^n$. We use $\EXP{z}$ to denote the expectation of a random variable~$z$ and $Prob(A)$ to denote the probability of an event $A$. We let SOL$(X,F)$ denote the solution set of problem \eqref{eqn:SCVI}. 
\section{Preliminaries}\label{sec:pre}
In this section, we \fyR{provide} the background for the prox mappings and review some of their main properties. More \fyR{details} on prox mappings can be found in \cite{Nemir2000}. 

Let $i \in \{1,\ldots,d\}$. A function $\omega_i:X_i \to \Real$ is called a distance generating function with modulus $\mu_{\omega_i}>0$ with respect to norm $\|\cdot\|_i$, if $\omega_i$ is a continuously differentiable and strongly convex function with parameter $\mu_{\omega_i}$ with respect to $\|\cdot\|_i$, i.e.,  
\begin{align}
\omega_i(y) \geq \omega_i(x) + \langle \nabla \omega_i(x), y-x\rangle +\frac{\mu_{\omega_i}}{2}\|x-y\|_i^2 \quad \hbox{for all } x,y \in X_i.
\end{align}
Throughout the paper, we assume that the function $\omega_i$ has Lipschitz continuous gradients with parameter $L_{\omega_i}$, i.e., 
\begin{align}
\omega_i(y)\leq \omega_i(x)+\langle \nabla \omega_i(x),y-x \rangle +\frac{L_{\omega_i}}{2}\|x-y\|_i^2 \quad \hbox{for all } x,y \in X_i.
\end{align}
The Bregman distance function (also called prox function) $D_i: X_i\times X_i \to \Real$ associated with $\omega_i$ is defined as follows:
\begin{align}\label{def:Di}
D_i(x,y) =\omega_i(y)-\omega_i(x)-\langle \nabla \omega_i(x),y-x \rangle \quad \hbox{for all } x,y \in X_i.
\end{align}
Next we define the prox mapping $P_i:X_i\times\Real^{ n_i}\to X_i$ as follows:
\fyR{\begin{align}
P_i(x,y)=\argmin_{z\in X_i}\{\langle y, z\rangle+D_i(x,z)\}, \quad \hbox{for all }x \in X_i, \ y \in \Real^{n_i}. 
\end{align}}
The next Lemma provides some of the main properties of Bregman functions and their associated prox mappings. The proof of these properties can be found in earlier work by Nemirovski et al., e.g., see Chapter 5 in \cite{Nemir2000}, and also \cite{nemirovski_robust_2009}.
\begin{lemma}[Properties of prox mappings]\label{lemma:proxprop}Let $i \in \{1,\ldots,d\}$ be given. The following hold:
\begin{itemize}
\item [(a)]  The Bregman distance function $D_i$ satisfies the following inequalities:
\begin{align}
\frac{\mu_{\omega_i}}{2} \|x-y\|_i^2\leq D_i(x,y) \leq \frac{L_{\omega_i}}{2}\|x-y\|_i^2 \quad \hbox{for all } x,y \in X_i.
\end{align}
\item [(b)] For any $x,z \in X_i$ and $y \in \Real^{n_i}$, we have  
\[D_i(P_i(x,y),z) \leq D_i(x,z)+\langle y,z-P_i(x,y)\rangle-D_i(x,P_i(x,y)).\]
\item [(c)] For any $x,z \in X_i$ and $y \in \Real^{n_i}$, we have  
\[D_i(P_i(x,y),z) \leq D_i(x,z)+\langle y,z-x\rangle+\frac{{\|y\|_{*i}}^2}{2\mu_{\omega_i}}.\]
\item [(d)] For any $x \in X_i$, we have $x=P_i(x,0)$.
\item [(e)] The mapping $P_i(x,y)$is Lipschitz continuous in $y$ with modulus 1, i.e., \[\|P_i(x,y) -P_i(x,z)\|_i \leq \|y-z\|_{*i} \qquad\hbox{for all }x \in X_i, \ y,z \in \Real^{n_i}.\]
\item [(f)] For any $x,y,z \in X_i$, we have $D_i(x,z)=D_i(x,y)+D_i(y,z)+\langle \nabla \omega_i(y)-\nabla \omega_i(x),z-y\rangle.$
\item [(g)] For any $x,z \in\Real^{n_i}$, we have $\nabla_z D_i(x,z) = \nabla \omega_i(z)-\nabla \omega_i(x)$.
\end{itemize}
\end{lemma}
\fyR{In the analysis of the algorithms, we make use of} the following result, which can be found in~\cite{Polyak87} on page 50.
\begin{lemma}\label{lemma:supermartingale}
Let $v_k,$ $u_k,$ $\alpha_k,$ and  $\beta_k$ be
non-negative random variables, and let the
following relations hold almost surely: 
\[\EXP{v_{k+1}\mid {\tilde \sF_k}} 
\le (1+\alpha_k)v_k - u_k + \beta_k \quad\hbox{ for all } k,\qquad
\sum_{k=0}^\infty \alpha_k < \infty,\qquad
\sum_{k=0}^\infty \beta_k < \infty,\] 
where $\tilde \sF_k$ denotes the collection $v_0,\ldots,v_k$, $u_0,\ldots,u_k$,
$\alpha_0,\ldots,\alpha_k$, $\beta_0,\ldots,\beta_k$. 
Then, almost surely we have 
\[\lim_{k\to\infty}v_k = v, \qquad \sum_{k=0}^\infty u_k < \infty,\]
where $v \geq 0$ is some random variable.
\end{lemma}
Next, we recall the following definitions (cf. \cite{facchinei02finite}) that will be referred to in our analysis.
\begin{definition}[Types of monotonicity]\label{def:map} Consider a mapping $F:X\to \Real^n$. 
\begin{itemize}
\item [(a)] $F$ is called a monotone mapping if for any $x,y \in X$, we have $\langle F(x)-F(y),x-y \rangle \geq 0$.
\item [(b)] $F$ is called a strictly monotone mapping if for any $x,y \in X$, we have $\langle F(x)-F(y),x-y \rangle > 0$.
\item [(c)] $F$ is called a $\mu$-strongly monotone mapping if there is $\mu>0$ such that for any $x,y \in X$, we have \[\langle F(x)-F(y),x-y \rangle \geq \mu\|x-y\|^2.\]
\item [(d)]  $F$ is called a pseudo-monotone mapping if for any \fyR{$x,y \in X$}, $\langle F(y),x-y \rangle \geq 0$ implies that \[\langle F(x),x-y \rangle \geq 0.\]
\item [(e)] $F$ is called a strictly pseudo-monotone mapping if for \fyR{any $x,y \in X$  and $x\neq y$}, $\langle F(y),x-y \rangle \geq 0$ implies that \[\langle F(x),x-y \rangle > 0.\]
\item [(f)] $F$ is called a $\mu$-strongly pseudo-monotone mapping if there is $\mu>0$ such that for any $x,y \in X$, $\langle F(y),x-y \rangle \geq 0$ implies that \[\langle F(x),x-y \rangle \geq \mu\|x-y\|^2.\]
\end{itemize}
\end{definition} 
It is worth noting that $(a) \Rightarrow (d)$, $(b) \Rightarrow (e)$, and $(c) \Rightarrow (f)$.

\section{Randomized block stochastic mirror-prox algorithm: assumptions and algorithm outline}\label{sec:alg}
In this section, we outline the randomized block stochastic mirror-prox (B-SMP) algorithm and state the main assumptions. Algorithm \ref{algorithm:IRLSA-impl} shows the steps of the B-SMP method. At iteration $k$, first a \fyR{realization
of} random variable $i_k$ is generated from the probability distribution $P_b$, where $Prob(i_k=i)=p_i$. Throughout the analysis, we assume $p_i>0$ for all $i=1,\ldots,d$ and $\sum_{i=1}^dp_i=1$. Step 4 in Algorithm \ref{algorithm:IRLSA-impl} provides the update rules. First, using a stochastic oracle, a \fyR{realization} of the stochastic mapping $F(\cdot,\xi)$ is generated at $x_k$, denoted by $F(x_k,\tilde \xi_k)$. Next, the $i_k$th block of the vector $y_{k}$ is updated using the projection given by \eqref{BSMPproj1}. Here $\g_k>0$ is the stepsize sequence. The stochastic oracle is called at the resulting vector $y_{k+1}$ and $F(y_{k+1},\xi_k)$ is generated. Using the block $F_{i_k}(y_{k+1},\xi_k)$, stepsize $\g_k$, and the value of $x_k^{i_k}$, the $i_k$th block of $x_k$ is updated using the rule \eqref{BSMPproj2}. We note that the B-SMP method extends the SMP algorithm proposed in \cite{Nem11} to a block coordinate variant. This results in reducing the computational complexity of each iteration substantially by only performing the two projection steps on a component set, versus two projections on the entire set $X \in \Real^n$. The B-SMP algorithm differs from the SBMD method developed in \cite{LanSBMD} mainly from two aspects: (i) While in step 4, we do two projections similar to the extragradient method, in SBMD, a single projection is performed; (ii) Unlike the SBMD algorithm, here we do not use averaging. Later, we also study the averaging variant of the B-SMP method for stochastic convex optimization problems (see Proposition \ref{prop:optaveSCO}). 
\begin{algorithm}
  \caption{Randomized block stochastic mirror-prox (B-SMP) algorithm}
\label{algorithm:IRLSA-impl}
    \begin{algorithmic}[1]
    \STATE \textbf{initialization:} Set $k=0$, a random initial point $x_0=y_0\in X$, a stepsize $\g_0>0$, a discrete probability distribution $P_b$ with probabilities $p_i$ for $i=1,\ldots,d$ and a scalar $K$; 
    \FOR {$k=0,1,\ldots,K-1$}
        \STATE Generate a \fyR{realization of} random variable $i_k$ using the probability distribution $P_b$ such that $Prob(i_k=i)=p_i$; \fyR{Generate $\xi_k$ and $\tilde \xi_k$ as realizations of the random vector $\xi$;}
         \STATE Update the blocks $x_k^{i_k}$ and $y_k^{i_k}$:
\begin{align}
y_{k+1}^{i}&:= \left\{\begin{array}{ll}P_{i_k}\left(x_{k}^{i_k},\g_k F_{i_k}(x_{k},\tilde \xi_k)\right)
&\hbox{if } i=i_k,\cr \hbox{} &\hbox{}\cr
x_k^{i}& \hbox{if } i\neq i_k,\end{array}\right.\label{BSMPproj1}
\\
x_{k+1}^{i}&:=\left\{\begin{array}{ll}P_{i_k}\left(x_{k}^{i_k},\g_k F_{i_k}(y_{k+1},\xi_k)\right)
&\hbox{if } i=i_k,\cr \hbox{} &\hbox{}\cr
x_k^{i}& \hbox{if } i\neq i_k.\end{array}\right.\label{BSMPproj2}
\end{align}
    \ENDFOR
    \STATE return $x_{K};$
   \end{algorithmic}
\end{algorithm}

Throughout, we let $\sF_{k}$ denote the history of the method up to the $k$th iteration, i.e., \[\sF_{k}\triangleq\{i_0,\tilde \xi_0,\xi_0,\ldots,i_{k},\tilde \xi_{k},\xi_{k}\}.\]
Recall $F(x)$ denotes the expected value of the stochastic mapping $F(x,\xi)$ where $\xi \in \Real^m$, i.e., $F(x)=\EXP{F(x,\xi)}$. We define the stochastic errors $\tilde w_k$ and $w_k$ as follows:
\begin{align}\label{def:wk}
\tilde w_k\triangleq F(x_k,\tilde \xi_k)-F(x_k),  \qquad w_k\triangleq F(y_{k+1},\xi_k)-F(y_{k+1}).
\end{align}
In the following, we state the main assumptions that will be considered in our analysis in the next sections.
\begin{assumption}[Problem]\label{assump:main}
\begin{itemize}
\item [(a)] For any $i \in \{1,\ldots,d\}$, the set $X_i$ is nonempty, closed, convex, and bounded, i.e., there exists $B_i \in \Real$ such that $\|x\|_{i} \leq B_i$ for all $x\in X_i$. 
\fyR{\item [(b)] For any $i$,  the mapping $F_i$ is block-Lipschitz continuous with parameter $L_i$, i.e., there exists $L_i \in \Real$ such that 
\[\|F_i(x^i;x^{-i})-F_i(y^i;x^{-i})\|_{*i}\leq L_i\|x^i-y^i\|_i, \quad \hbox{for all } x^i, y^i \in X_i \hbox{ and }x^{-i} \in \prod_{j=1,j\neq i}^d X_j.\]}
\end{itemize}
\end{assumption}
\fyR{\begin{remark}\label{rem1}
Note that the block-Lipschitzian property in Assumption \ref{assump:main}(b) implies continuity of the mapping $F$. Compactness, and convexity of the set $X$, along with continuity of the mapping $F$ imply that SOL$(X,F)$ is nonempty and compact (cf. Corollary 2.2.5 in~\cite{facchinei02finite}). Also, compactness of the sets $X_i$ in Assumption \ref{assump:main}(a) and continuity of mapping $F$ in Assumption \ref{assump:main}(b) imply boundedness of the mapping $F$. Therefore, for any $i$, the mapping $F_i$ is bounded on $X$, i.e., there exists $C_i \in \Real_+$ such that $\|F_i(x)\|_{*i}\leq C_i \hbox{ for all } x \in X.$ Throughout the paper, $C_i$ denotes the bound on $F_i(x)$ with respect to the norm $\|\cdot\|_{*i}$.
\end{remark}
}
\begin{assumption}{\bf{(Random variables)}}\label{assump:randvar}
\begin{itemize}
\item [(a)] The random variables $\xi_k$ and $\tilde \xi_k$, and $i_k$ are all i.i.d. and independent of each other.
\item [(b)] \fyR{For all $k\geq 0$, the stochastic mappings $F(\cdot,\tilde \xi_k)$ and $F(\cdot,\xi_k)$ are unbiased estimators of the mapping $F(\cdot)$, i.e., for all $x \in X$}
\[\fyR{\EXP{F(x,\tilde \xi_k)\mid x}=\EXP{F(x,\xi_k)\mid x}=F(x).}\]
\item [(c)] \fyR{The stochastic mappings $F(\cdot,\tilde \xi_k)$ and $F(\cdot,\xi_k)$ both have bounded variances. More precisely, for all $i$ and $k\geq 0$, there exist positive scalars $\nu_i$ and $\tilde \nu_i$ such that}
\[\fyR{\EXP{\|F_i(x,\tilde \xi_k)-F_i(x)\|_{*i}^2\mid x}\leq {\tilde \nu_i}^2, \quad \EXP{\|F_i(x,\xi_k)-F_i(x)\|_{*i}^2\mid x}\leq {\nu_i}^2.}\]
\item [(d)] For all $k\geq 0$, the random variable $i_k$ is drawn from a discrete probability distribution $P_b$ \fyR{where $Prob(i_k=i)=p_i$ and $p_i>0$.}
\end{itemize}
\end{assumption}
\fyR{\begin{remark}
Note that $x_k$ is $\sF_{k-1}$-measurable, and $y_{k+1}$ is $\sF_{k-1}\cup \{\tilde \xi_k, i_k\}$-measurable. Therefore, from the definition of the stochastic errors $\fyR{\tilde w_k}$ and $w_k$ in \eqref{def:wk}, Assumption \ref{assump:randvar}(b) implies that for all $i=1,\ldots,d$, we have for all $k\geq 0$
\[\EXP{{\fyR{\tilde w_k}}^i\mid \sF_{k-1}}=0, \quad \EXP{{w_k}^i\mid \sF_{k-1}\cup \{\tilde \xi_k,i_k\}}=0.\]
Similarly, from Assumption \ref{assump:randvar}(c) we obtain for all $i=1,\ldots,d$, and for all $k \geq 0$
\[ \EXP{\|{\fyR{\tilde w_k}}^i\|_{*i}^2\mid \sF_{k-1}}\leq {\fyR{\tilde \nu}^2_i}, \quad \EXP{\|w_k^i\|_{*i}^2\mid \sF_{k-1}\cup \{\tilde \xi_k,i_k\}}\leq \nu_i^2.\]
\end{remark}}
\begin{assumption}{\bf{(Stepsize sequence)}}\label{assump:stepsize} Let the \fyR{positive} stepsize sequence $\{\g_k\}$ satisfy the following:
\begin{itemize}
\item [(a)] $\g_k$ is square summable, i.e., $\sum_{k=0}^\infty \g_k^2 < \infty$.
\item [(b)] $\g_k$ is non-summable, i.e., $\sum_{k=0}^\infty \g_k = \infty$.
\end{itemize}
\end{assumption}
\section{Convergence analysis of the B-SMP algortihm}\label{sec:conv}
In this section, we establish the convergence of the B-SMP algorithm in an almost sure sense. 
Throughout the analysis, we make use of a Lyapunov function $\mathcal{L}:X\times \Real^n\to \Real$ given by \begin{align}\label{def:L}\mathcal{L}(x,y)\triangleq\sum_{i=1}^dp_i^{-1}D_i(x^i,y^i) \qquad \hbox{for any }x \in X \hbox{ and } y \in \Real^n.\end{align} 
The following result will be used in our analysis.
\begin{lemma}[Properties of $\mathcal{L}$]\label{lemma:lyapProp}
\fyR{Let Assumption \ref{assump:randvar}(d) hold.} Function $\mathcal{L}$ satisfies the following:
\begin{itemize}
\item [(a)] $\mathcal{L}(x,y)\geq 0$ for any $x \in X$ and $y \in \Real^n$.
\item [(b)]$\mathcal{L}(x,y)=0$ holds if and only if $x=y$.
\end{itemize}
\end{lemma}
\begin{proof}
\noindent (a) This is an immediate consequence of Lemma \ref{lemma:proxprop}(a).\\
\noindent (b) Note that from the definition of $D_i$, for any $i$, we have $D_i(x^i,x^i)=0$ for any $x^i \in X_i$. This implies that $\mathcal{L}(x,x)=0$. Let us assume $\mathcal{L}(x,y)=0$ for some $x \in X$ and $y \in \Real^n$. From non-negativity of $D_i$, and the definition of $\mathcal{L}$, for any $i=1,\ldots, d$ we have $D_i(x^i,y^i)=0$. Invoking Lemma \ref{lemma:proxprop}(a) again, we have $\|x^i-y^i\|_i=0$. Since $\|\cdot\|_i$ is a norm operator, we have $x^i=y^i$ for all $i=1,\ldots, d$, implying that $x=y$.
\end{proof}
The first result stated below, provides a recursive relation in terms of the expected value of the Lyapunov function. This result will be used in both convergence analysis in this section and deriving the rate statements in the subsequent section.
\begin{lemma}[A recursive relation for the Lyapunov function]\label{lemma:rec-bound}
Let Assumptions~\ref{assump:main} and~\ref{assump:randvar} hold.
Then, for any $k\geq 0$, for the iterate $x_k$ generated by Algorithm \ref{algorithm:IRLSA-impl},
the following relation holds for all $x \in X$:
 \begin{align}\label{ineq:rec-bound}
\EXP{\mathcal{L}(x_{k+1},x)\mid \sF_{k-1}}&\leq \mathcal{L}(x_k,x)+\g_k\langle F(x_k), x-x_k\rangle+\sum_{i=1}^d\left(\frac{C_{i}^2+\nu_i^2}{\mu_{\omega_{i}}}+2L_{i}B_{i}(C_{i}+\tilde \nu_i)\right)\g_k^2.
\end{align}
\end{lemma}
\begin{proof}
Consider the relation $x_{k+1}^{i_k}=P_{i_k}\left(x_{k}^{i_k},\g_kF_{i_k}(y_{k+1},\xi_{k})\right)$. From Lemma \ref{lemma:proxprop}(c), for an arbitrary vector $x \in X$, we obtain 
\begin{align}\label{ineq:Di}
D_{i_k}(x_{k+1}^{i_k},x^{i_k})& \leq D_{i_k}(x_{k}^{i_k},x^{i_k})+\g_k\langle F_{i_k}(y_{k+1},\xi_k), x^{i_k}-x_k^{i_k}\rangle+\g_k^2\frac{\|F_{i_k}(y_{k+1},\xi_k)\|_{*i_k}^2}{2\mu_{\omega_{i_k}}}\notag\\
& \leq D_{i_k}(x_{k}^{i_k},x^{i_k})+\underbrace{\g_k\langle F_{i_k}(y_{k+1}), x^{i_k}-x_k^{i_k}\rangle}_{\mbox{Term}1}+\g_k\langle w_k^{i_k}, x^{i_k}-x_k^{i_k}\rangle\notag\\
&+\g_k^2\frac{\|F_{i_k}(y_{k+1})\|_{*i_k}^2}{\mu_{\omega_{i_k}}}+\g_k^2\frac{\|w_k^{i_k}\|_{*i_k}^2}{\mu_{\omega_{i_k}}},
\end{align}
where we used the triangle inequality for the dual norm $\|\cdot\|_{*i_k}$ in the preceding inequality. Next, we provide an upper bound for Term $1$. Adding and subtracting $F_{i_k}(x_k)$ and using the Lipschitzian property of the mapping $F_{i_k}$, we obtain
\begin{align}\label{ineq:Term1}
\mbox{Term}\ 1&= \g_k\langle F_{i_k}(y_{k+1})-F_{i_k}(x_k), x^{i_k}-x_k^{i_k}\rangle+\g_k\langle F_{i_k}(x_k), x^{i_k}-x_k^{i_k}\rangle\notag\\
& \leq \g_k\|F_{i_k}(y_{k+1})-F_{i_k}(x_k)\|_{*i_k}\|x^{i_k}-x_k^{i_k}\|_{i_k}+\g_k\langle F_{i_k}(x_k), x^{i_k}-x_k^{i_k}\rangle\notag\\
& \leq 2B_{i_k}\g_k\underbrace{\|F_{i_k}(y_{k+1})-F_{i_k}(x_k)\|_{*i_k}}_{\mbox{Term}\ 2}+\g_k\langle F_{i_k}(x_k), x^{i_k}-x_k^{i_k}\rangle,
\end{align}
where in the first inequality, we use the definition of the dual norm $\|\cdot\|_{*i}$, while in the preceding relation, we use the triangle inequality for the term $\|x^{i_k}-x_k^{i_k}\|_{i_k}$ and boundedness of the set $X_{i_k}$. Next, we estimate an upper bound on $\mbox{Term}\ 2$. We have 
\begin{align*}
\mbox{Term}\ 2&= \|F_{i_k}(y_{k+1})-F_{i_k}(x_k)\|_{*i_k}=\|F_{i_k}(y_{k+1}^{i_k};x_{k}^{-i_k})-F_{i_k}(x_{k}^{i_k};x_{k}^{-i_k})\|_{*i_k}\leq L_{i_k}\|y_{k+1}^{i_k}-x_{k}^{i_k}\|_{i_k},
\end{align*}
where we use the block-Lipschitzian property of mapping $F$ given by Assumption \ref{assump:main}(b). From Lemma \ref{lemma:proxprop}(d,e) and the update rule for $y_{k+1}$, we obtain
\begin{align*}
\mbox{Term}\ 2& \leq L_{i_k}\|y_{k+1}^{i_k}-x_{k}^{i_k}\|_{i_k}=L_{i_k}\left\|P_{i_k}\left(x_{k}^{i_k},\g_k F_{i_k}(x_{k},\tilde \xi_k)\right)-P_{i_k}\left(x_{k}^{i_k},0\right)\right\|_{i_k} \leq L_{i_k}\|\g_k F_{i_k}(x_{k},\tilde \xi_k)\|_{*i_k}\\
& =\g_kL_{i_k}\| F_{i_k}(x_{k})+\fyR{{\tilde w}^{i_k}_k}\|_{*i_k}\leq \g_kL_{i_k}\| F_{i_k}(x_{k})\|_{*i_k}+\g_kL_{i_k}\|\fyR{{\tilde w}^{i_k}_k}\|_{*i_k}.
\end{align*}
Therefore, from \eqref{ineq:Term1} and Assumption \ref{assump:main}(c), we obtain
\begin{align*}
\mbox{Term}\ 1& \leq  2L_{i_k}B_{i_k}\g_k^2\| F_{i_k}(x_{k})\|_{*i_k}+2L_{i_k}B_{i_k}\g_k^2\|\fyR{{\tilde w}^{i_k}_k}\|_{*i_k}+\g_k\langle F_{i_k}(x_k), x^{i_k}-x_k^{i_k}\rangle \\
& \leq 2L_{i_k}B_{i_k}C_{i_k}\g_k^2+2L_{i_k}B_{i_k}\g_k^2\|\fyR{{\tilde w}^{i_k}_k}\|_{*i_k}+\g_k\langle F_{i_k}(x_k), x^{i_k}-x_k^{i_k}\rangle.
\end{align*}
From the preceding inequality, \eqref{ineq:Di}, and Assumption \ref{assump:main}(c), we have 
\begin{align*}
D_{i_k}(x_{k+1}^{i_k},x^{i_k})& \leq D_{i_k}(x_{k}^{i_k},x^{i_k})+\g_k\langle F_{i_k}(x_k), x^{i_k}-x_k^{i_k}\rangle+\g_k\langle w_k^{i_k}, x^{i_k}-x_k^{i_k}\rangle\notag\\
&+\g_k^2\left(\frac{C_{i_k}^2}{\mu_{\omega_{i_k}}}+2L_{i_k}B_{i_k}C_{i_k}\right)+\g_k^2\frac{\|w_k^{i_k}\|_{*i_k}^2}{\mu_{\omega_{i_k}}}+2L_{i_k}B_{i_k}\g_k^2\|\fyR{{\tilde w}^{i_k}_k}\|_{*i_k}.
\end{align*}
Note that by the definition of Algorithm \ref{algorithm:IRLSA-impl}, for any $i\neq i_k$, we have $x_{k+1}^i=x_k^i$. Thus, from the definition of $\mathcal{L}$ given by \eqref{def:L}, and the preceding relation we have
\begin{align*}
\mathcal{L}(x_{k+1},x)&=\sum_{i\neq i_k}p_i^{-1}D_i(x_{k+1}^i,x^i)+p_{i_k}^{-1}D_{i_k}(x_{k+1}^{i_k},x^{i_k})\\
&=\sum_{i\neq i_k}p_i^{-1}D_i(x_{k}^i,{x}^i)+p_{i_k}^{-1}D_{i_k}(x_{k+1}^{i_k},{x}^{i_k})\\
& \leq \mathcal{L}(x_k,x)+p_{i_k}^{-1}\g_k\langle F_{i_k}(x_k), x^{i_k}-x_k^{i_k}\rangle+p_{i_k}^{-1}\g_k\langle w_k^{i_k}, x^{i_k}-x_k^{i_k}\rangle\\
&+p_{i_k}^{-1}\g_k^2\left(\frac{C_{i_k}^2}{\mu_{\omega_{i_k}}}+2L_{i_k}B_{i_k}C_{i_k}\right)+p_{i_k}^{-1}\g_k^2\frac{\|w_k^{i_k}\|_{*i_k}^2}{\mu_{\omega_{i_k}}}+2p_{i_k}^{-1}L_{i_k}B_{i_k}\g_k^2\|\fyR{{\tilde w}^{i_k}_k}\|_{*i_k}.
\end{align*}
\fyR{Next, we take conditional expectations from the preceding relation on $\sF_{k-1}\cup \{\tilde \xi_k, i_k\}$. Note that in Algorithm~\ref{algorithm:IRLSA-impl}, $x_k$ is $\sF_{k-1}$-measurable, and $y_{k+1}$ and $\fyR{\tilde w_k}$ are both $\sF_{k-1}\cup \{\tilde \xi_k, i_k\}$-measurable. Therefore, we obtain
\begin{align*}
\EXP{\mathcal{L}(x_{k+1},x)\mid \sF_{k-1}\cup \{\tilde \xi_k, i_k\}}&\leq \mathcal{L}(x_k,x)+p_{i_k}^{-1}\g_k\langle F_{i_k}(x_k), x^{i_k}-x_k^{i_k}\rangle\\
&+\EXP{p_{i_k}^{-1}\g_k\langle w_k^{i_k}, x^{i_k}-x_k^{i_k}\rangle\mid \sF_{k-1}\cup \{\tilde \xi_k, i_k\}}\\& +p_{i_k}^{-1}\g_k^2\left(\frac{C_{i_k}^2}{\mu_{\omega_{i_k}}}+2L_{i_k}B_{i_k}C_{i_k}\right)+\EXP{p_{i_k}^{-1}\g_k^2\frac{\|w_k^{i_k}\|_{*i_k}^2}{\mu_{\omega_{i_k}}}\mid \sF_{k-1}\cup \{\tilde \xi_k, i_k\}}\\& +2p_{i_k}^{-1}L_{i_k}B_{i_k}\g_k^2\|\fyR{{\tilde w}^{i_k}_k}\|_{*i_k}.
\end{align*}
Rearranging the terms, we obtain 
\begin{align*}
\EXP{\mathcal{L}(x_{k+1},x)\mid \sF_{k-1}\cup \{\tilde \xi_k, i_k\}}&\leq \mathcal{L}(x_k,x)+p_{i_k}^{-1}\g_k\langle F_{i_k}(x_k), x^{i_k}-x_k^{i_k}\rangle\\
&+p_{i_k}^{-1}\g_k\langle \EXP{w_k^{i_k}\mid \sF_{k-1}\cup \{\tilde \xi_k, i_k\}}, x^{i_k}-x_k^{i_k}\rangle\\& +p_{i_k}^{-1}\g_k^2\left(\frac{C_{i_k}^2}{\mu_{\omega_{i_k}}}+2L_{i_k}B_{i_k}C_{i_k}\right)+p_{i_k}^{-1}\g_k^2\frac{\EXP{\|w_k^{i_k}\|_{*i_k}^2\mid \sF_{k-1}\cup \{\tilde \xi_k, i_k\}}}{\mu_{\omega_{i_k}}}\\& +2p_{i_k}^{-1}L_{i_k}B_{i_k}\g_k^2\|\fyR{{\tilde w}^{i_k}_k}\|_{*i_k}.
\end{align*}
Invoking Assumption \ref{assump:randvar}, we can write
\begin{align}\label{ineq:rev1ineq1}
\EXP{\mathcal{L}(x_{k+1},x)\mid \sF_{k-1}\cup \{\tilde \xi_k, i_k\}}&\leq \mathcal{L}(x_k,x)+p_{i_k}^{-1}\g_k\langle F_{i_k}(x_k), x^{i_k}-x_k^{i_k}\rangle\notag\\
& +p_{i_k}^{-1}\g_k^2\left(\frac{C_{i_k}^2}{\mu_{\omega_{i_k}}}+2L_{i_k}B_{i_k}C_{i_k}\right)+p_{i_k}^{-1}\g_k^2\frac{\nu_{i_k}^2}{\mu_{\omega_{i_k}}}\notag\\& +\underbrace{2p_{i_k}^{-1}L_{i_k}B_{i_k}\g_k^2\|\fyR{{\tilde w}^{i_k}_k}\|_{*i_k}}_{\mbox{Term}\ 3}.
\end{align}
In the remainder of the proof, the idea is to first take conditional expectations from the preceding inequality with respect to $\tilde \xi_k$, and then with respect to $i_k$. Note that for Term $3$, we can write
\begin{align*}
\EXP{\mbox{Term}\ 3\mid \sF_{k-1}\cup \{i_k\}}&\leq 2\g_k^2L_{i_k}B_{i_k}\EXP{\|\fyR{{\tilde w}^{i_k}_k}\|_{*i_k}\mid \sF_{k-1}\cup \{i_k\}}
\\ &\leq 2\g_k^2L_{i_k}B_{i_k}\sqrt{\EXP{\|{\fyR{\tilde w_k}}^{i_k}\|^2_{*{i_k}}\mid \sF_{k-1}\cup \{i_k\}}}\leq2\g_k^2L_{i_k}B_{i_k}\fyR{{\tilde\nu}_{i_k}},\end{align*}
where in the second inequality, we applied Jensen's inequality, and in the last inequality we used Assumption \ref{assump:randvar}. 
Taking expectations in relation \eqref{ineq:rev1ineq1} with respect to $\tilde \xi_k$, and using the preceding estimate, we obtain
\begin{align*}
\EXP{\mathcal{L}(x_{k+1},x)\mid \sF_{k-1}\cup \{i_k\}}&\leq \mathcal{L}(x_k,x)+p_{i_k}^{-1}\g_k\langle F_{i_k}(x_k), x^{i_k}-x_k^{i_k}\rangle\\
& +p_{i_k}^{-1}\g_k^2\left(\frac{C_{i_k}^2}{\mu_{\omega_{i_k}}}+2L_{i_k}B_{i_k}C_{i_k}\right)+p_{i_k}^{-1}\g_k^2\frac{\nu_{i_k}^2}{\mu_{\omega_{i_k}}}+2\g_k^2L_{i_k}B_{i_k}\fyR{{\tilde\nu}_{i_k}}.
\end{align*}
Taking expectations in the preceding inequality with respect $i_k$, we have 
\begin{align}\label{ineq2:L}
\EXP{\mathcal{L}(x_{k+1},x)\mid \sF_{k-1}}&\leq \mathcal{L}(x_k,x)+\sum_{i=1}^dp_{i}p_{i}^{-1}\g_k\langle F_{i}(x_k), x^{i}-x_k^{i}\rangle\notag\\
&+\sum_{i=1}^dp_{i}p_{i}^{-1}\g_k^2\left(\frac{C_{i}^2}{\mu_{\omega_{i}}}+2L_{i}B_{i}C_{i}\right)+\sum_{i=1}^dp_{i}p_{i}^{-1}\g_k^2\frac{\nu_{i}^2}{\mu_{\omega_{i}}}\notag\\ &+2\sum_{i=1}^dp_{i}p_{i}^{-1}L_{i}B_{i}\g_k^2\fyR{{\tilde \nu}_{i}}.
\end{align}
From the definition of the inner product, by rearranging the terms, we obtain the desired inequality.}
\end{proof}
Before we proceed to establish the almost sure convergence for Algorithm \ref{algorithm:IRLSA-impl}, in the following result, we present the properties of Algorithm \ref{algorithm:IRLSA-impl} when the mapping $F$ is pseudo-monotone.
\begin{lemma}[Properties under pseudo-monotonicity]\label{lemma:pseudo}
Consider problem \eqref{eqn:SCVI} where the mapping $F$ is pseudo-monotone on $X$. Let Assumptions \ref{assump:main}, \ref{assump:randvar}, and \ref{assump:stepsize} hold. Let $\{x_k\}$ be generated by Algorithm \ref{algorithm:IRLSA-impl}. Then, the following results hold almost surely:
\begin{itemize}
\item [(a)] For any $x^* \in \hbox{SOL}(X,F)$, the sequence $\{\mathcal{L}(x_k,x^*)\}$ is convergent.
\item [(b)] The sequence $\{x_k\}$ is bounded.
\item [(c)] For any $x^* \in \hbox{SOL}(X,F)$, $\liminf_{k \to \infty} \langle F(x_k), x^*-x_k\rangle=0$.
\item [(d)] If an accumulation point of $\{x_k\}$ is a solution to problem \eqref{eqn:SCVI}, then the \fyR{entire} sequence $\{x_k\}$ converges to this solution.
\end{itemize}
\end{lemma}
\begin{proof}
Consider relation \eqref{ineq:rec-bound}. Let $x=x^*$ where $x^* \in \hbox{SOL}(X,F)$. We have
 \begin{align*}
\EXP{\mathcal{L}(x_{k+1},x^*)\mid \sF_{k-1}}&\leq \mathcal{L}(x_k,x^*)-\g_k\langle F(x_k), x_k-x^*\rangle+\sum_{i=1}^d\left(\frac{C_{i}^2+\nu_i^2}{\mu_{\omega_{i}}}+2L_{i}B_{i}(C_{i}+\tilde \nu_i)\right)\g_k^2.
\end{align*}
Next, we apply Lemma \ref{lemma:supermartingale}. Let us define \[\alpha_k=0, \quad \beta_k=\sum_{i=1}^d\left(\frac{C_{i}^2+\nu_i^2}{\mu_{\omega_{i}}}+2L_{i}B_{i}(C_{i}+\tilde \nu_i)\right)\g_k^2, \quad u_k=\g_k\langle F(x_k), x_k-x^*\rangle,\quad v_k=\mathcal{L}(x_k,x^*).\] From Assumption \ref{assump:stepsize}(a), we have $\sum_{k=0}^\infty \beta_k <\infty$. Moreover, since $x^*$ is a solution to the problem \eqref{eqn:SCVI}, we have $\langle F(x^*), x_k-x^*\rangle \geq 0$. Invoking the pseudo-monotonicity property of $F$, we obtain $u_k \geq 0$. Therefore, since all the conditions of Lemma \ref{lemma:supermartingale} are met, we can conclude that the sequence $\{\mathcal{L}(x_k,x^*)\}$ is convergent almost surely, implying that part (a) holds. Moreover, $\sum_{k=0}^\infty u_k < \infty$ holds almost surely. Since we assumed $\sum_{k=0}^\infty\g_k=\infty$, we conclude that $\liminf_{k \to \infty} \langle F(x_k), x^*-x_k\rangle=0$ almost surely indicating that part (c) holds. Note that since the set $X$ is compact, the sequence $\{x_k\}$ is bounded implying that part (b) holds. Next we show part (d). Form the hypothesis, let $\{x_{k_j}\}$ denote the subsequence of $x_k$ where $\lim_{j\to \infty}x_{k_j}=\bar x \in \hbox{SOL}(X,F)$. From part (a), since $\bar x$ is a solution, the sequence $\{\mathcal{L}(x_k,\bar x)\}$ is convergent. Let $\bar{\mathcal{L}}$ denotes the limit point of $\{\mathcal{L}(x_k,\bar x)\}$, i.e., $\lim_{k\to \infty}\mathcal{L}(x_k,\bar x)=\bar{\mathcal{L}}$. Note that since each function $\omega_i$ is convex and therefore continuous, the function $\mathcal{L}$ is also continuous. Taking this into account, we have
\[\lim_{j\to \infty}\mathcal{L}(x_{k_j},\bar x)=\mathcal{L}\left(\lim_{j\to \infty}x_{k_j},\bar x\right)=\mathcal{L}(\bar x,\bar x)=0.\]
Therefore, $\bar{\mathcal{L}}=0$ implying that $\lim_{k\to \infty}\mathcal{L}(x_k,\bar x)=0$. Invoking continuity of $\mathcal{L}$ again and using Lemma \ref{lemma:lyapProp}, we obtain the desired result.
\end{proof}
The following lemma implies uniqueness of the solution of problem \eqref{eqn:SCVI} under strict pseudo-monotonicity of the mapping. 
\begin{lemma}\label{lem:uniqueness}
Consider problem \eqref{eqn:SCVI}. Let Assumption \ref{assump:main} hold and mapping $F$ be strictly pseudo-monotone on $X$. Then, \eqref{eqn:SCVI} has a unique solution.
\end{lemma}
\begin{proof}
See Appendix \ref{A1}.
\end{proof}
\fyR{Note that Lemma \ref{lemma:pseudo}(d) does not guarantee a.s. convergence to SOL$(X,F)$. To conclude this property, additional strict assumptions are needed. This is addressed in the following result.} 
\begin{proposition}[a.s. convergence under strict pseudo-monotonicity for SCVIs]\label{prop:a.s.vi} Consider problem \eqref{eqn:SCVI} and assume that $F$ is strictly pseudo-monotone on $X$. Let Assumptions \ref{assump:main}, \ref{assump:randvar}, and \ref{assump:stepsize} hold. Let $\{x_k\}$ be generated by Algorithm \ref{algorithm:IRLSA-impl}. Then, \fyR{$x_k$ converges to the unique solution of \eqref{eqn:SCVI}, $x^*$, almost surely.}
\end{proposition}
\fyR{\begin{proof}
The uniqueness of the solution of \eqref{eqn:SCVI} is implied by Lemma \ref{lem:uniqueness}. From Lemma \ref{lemma:pseudo}(c), \[\liminf_{k \to \infty} \langle F(x_k), x^*-x_k\rangle=0.\] Let us define the function $g:X\to \Real$ as $g(x)\triangleq\langle F(x), x^*-x\rangle$. We have $\liminf_{k \to \infty}g(x_k)=0$. This implies that there exists a subsequence $\{x_{k_j}\}$ (not necessarily convergent) such that $\lim_{j \to \infty}g(x_{k_j})=0$. From Lemma \ref{lemma:pseudo}(b), the sequence $\{x_k\}$ and therefore its subsequence $\{x_{k_j}\}$ are bounded. Thus, there exists a subsequence of the sequence $\{x_{k_j}\}$ that is convergent. Let us denote that subsequence by $\{x_{k_j(t)}\}$ and its accumulation point by $\hat x$. From $\lim_{j \to \infty}g(x_{k_j})=0$ we have $\lim_{t\to \infty}g(x_{k_j(t)})=0$. Using continuity of $g$, we obtain $g(\hat x)=0$ indicating that \begin{align}\label{ineq:spse4}\langle F(\hat x), x^*-\hat x\rangle=0.\end{align} 
Since $x^* \in $ SOL$(X,F)$, we have $\langle F(x^*),\hat x-x^*\rangle \ge 0$. Using the definition of strict pseudo-monotonicity of $F$ if $\hat x\neq x^*$, we obtain $\langle F(\hat x),\hat x-x^*\rangle > 0.$ This is contradictory to \eqref{ineq:spse4}. Therefore, $\hat x = x^*$. Thus, $x_k$ has an accumulation point $\hat x$ that solves VI$(X,F)$. Using Lemma \ref{lemma:pseudo}(d), we conclude that $x_k$ converges to $x^*$ almost surely. 
\end{proof}}
\section{Rate of convergence analysis for the B-SMP algorithm}\label{sec:rate}
To present the convergence rate of the B-SMP algorithm, in the first part of this section, we derive the rate under the assumption that the mapping $F$ is strongly pseudo-monotone. This result is provided by Proposition \ref{prop:ratestrongpse}. In the second part of this section, we consider a subclass of SCVIs, that is the stochastic convex optimization problems of the form \eqref{prob:SOP}. For this class of problems, we show that under convexity of the objective function, an averaging variant of Algorithm \ref{algorithm:IRLSA-impl} admits a convergence rate given by Proposition \ref{prop:optaveSCO}. Both of these results seem to be new for the stochastic mirror-prox algorithm addressing large-scale SCVIs. In the analysis of the first rate statement in this section, we make use of the following result. 
\begin{lemma}[Convergence rate of a recursive sequence]\label{lemma:rateHarmonic} Let $\{e_k\}$ be a non-negative sequence such that for an arbitrary non-negative sequence $\{\g_k\}$, the following relation is satisfied:
\begin{align}\label{ekbound0}
e_{k+1}\leq (1-\alpha \g_k)e_k+\beta \g_k^2, \quad \hbox{for all } k\geq 0.
\end{align}
where $\alpha$ and $\beta$ are positive scalars. Suppose $\g_0=\g$, $\g_k=\frac{\g}{k}$ for any $k\geq 1$, where $\g>\frac{1}{\alpha}$. Let $K\triangleq\lceil \alpha \g\rceil$. Then, we have 
\begin{align}\label{ekbound1}
e_k\leq \frac{\max \{\frac{\beta\g^2}{\alpha \g-1},Ke_K\}}{k}, \qquad \hbox{for all } k\geq K.
\end{align}
Specifically, if we set $\g=\frac{2}{\alpha}$, then we have 
\begin{align}\label{ekbound2}
e_k\leq \frac{8\beta}{\alpha^2k}, \qquad \hbox{for all } k\geq 2.
\end{align}
 \end{lemma}
\begin{proof}
See Appendix \ref{A3}.
\end{proof}
Next, we derive the convergence rate for the B-SMP method under when the mapping is strongly pseudo-monotone. 
\begin{proposition}[Rate statement under strong pseudo-monotonicity for SCVIs]\label{prop:ratestrongpse} 
Consider problem \eqref{eqn:SCVI} and assume that the mapping $F$ is $\mu$-strongly pseudo-monotone on $X$ with respect to the norm $\|\cdot\|$. Let Assumptions \ref{assump:main} and \ref{assump:randvar} hold. Let $\{x_k\}$ be generated by Algorithm \ref{algorithm:IRLSA-impl}. Then for $k\geq 0$, we have:
 \begin{align}\label{ineq:rec-bound3}
\EXP{\mathcal{L}(x_{k+1},{x^*})}&\leq \left(1-\g_k\frac{2\mu\min\limits_{1\leq i\leq d}p_i}{\max\limits_{1\leq i\leq d}L_{\omega_i}}\right)\EXP{\mathcal{L}(x_k,{x^*})}+\sum_{i=1}^d\left(\frac{C_{i}^2+\nu_i^2}{\mu_{\omega_{i}}}+2L_{i}B_{i}(C_{i}+\tilde \nu_i)\right)\g_k^2.
\end{align}
Let the probability distribution $P_b$ be uniform, i.e., $p_i=\frac{1}{d}$ for all $i \in \{1,\ldots,d\}$. Suppose the stepsize $\g_k$ is given by the following rule:
\begin{align}
\g_0\triangleq\frac{d\max\limits_{1\leq i\leq d}L_{\omega_i}}{\mu}, \qquad  \g_{k}=\frac{\g_0}{k}, \quad \hbox{for all }k \geq 1.\label{stepsize}
\end{align}
Then, $x_k$ converges to $x^*$ almost surely. Moreover, we have 
 \begin{align}\label{ineq:rec-bound4}
 \EXP{\|x_k-{x^*}\|^2}\leq \dfrac{\mathcal{A}d}{k}, \quad \hbox{for all } k\geq2,
\end{align}
where $\mathcal{A}\triangleq\frac{4\sum_{i=1}^d\left(\frac{C_{i}^2+\nu_i^2}{\mu_{\omega_{i}}}+2L_{i}B_{i}(C_{i}+\tilde \nu_i)\right)\left(\max\limits_{1\leq i\leq d}L_{\omega_i}\right)^2 }{\mu^2\min\limits_{1\leq i\leq d}\mu_{\omega_i}}$.
\end{proposition}
\begin{proof}
Since $F$ is strongly pseudo-monotone, similar to the first part of the proof of Proposition \ref{prop:a.s.vi}, it can be shown that the solution set is a singleton. Let $x^*$ denote the unique solution to problem \eqref{eqn:SCVI}. Consider relation \eqref{ineq:rec-bound}. \fyR{Taking expectations on both sides, we have}
 \begin{align}\label{ineq:rec-bound2}
\EXP{\mathcal{L}(x_{k+1},x^*)}&\leq \EXP{\mathcal{L}(x_k,x^*)}-\g_k\EXP{\langle F(x_k), x_k-x^*\rangle}+\sum_{i=1}^d\left(\frac{C_{i}^2+\nu_i^2}{\mu_{\omega_{i}}}+2L_{i}B_{i}(C_{i}+\tilde \nu_i)\right)\g_k^2.
\end{align}
Since $x^*$ solves problem \eqref{eqn:SCVI}, we have $\langle F(x^*),x_k-x^*\rangle \geq 0$. The definition of strong pseudo-monotonicity of $F$ implies that $\langle F(x_k),x_k-x^*\rangle \geq \mu\|x_k-x^*\|^2$. From the definition of the norm $\|\cdot\|$, we obtain
\begin{align*}
\langle F(x_k),x_k-x^*\rangle &\geq \mu\sum_{i=1}^d\|x_k^i-{x^*}^i\|^2_i \geq 2\mu\sum_{i=1}^d\frac{D_i(x_k^i,{x^*}^i)}{L_{\omega_{i}}}\\& \geq \frac{2\mu\min\limits_{1\leq i\leq d}p_i}{\max\limits_{1\leq i\leq d}L_{\omega_i}}\sum_{i=1}^d p_i^{-1}D_i(x_k^i,{x^*}^i)=\frac{2\mu\min\limits_{1\leq i\leq d}p_i}{\max\limits_{1\leq i\leq d}L_{\omega_i}}\mathcal{L}(x_k,x^*),
\end{align*}
where in the second inequality, we used the Lipschitzian property of the distance generator function $\omega_i$ (cf. Lemma \ref{lemma:proxprop}(a)). From the preceding relation, the definition of $\mathcal{L}$, and \eqref{ineq:rec-bound2}, we obtain the desired inequality \eqref{ineq:rec-bound3}. The almost sure convergence of the sequence $\{x_k\}$ to $x^*$ follows directly from Proposition \ref{prop:a.s.vi}. Next, we establish the rate statement. 
We apply the result of Lemma \ref{lemma:rateHarmonic} to the inequality \eqref{ineq:rec-bound3}. Since $P_b$ has a uniform distribution, we have $p_i=\frac{1}{d}$ for all $i$. Let us define 
\[e_k\triangleq\EXP{\mathcal{L}(x_{k},{x^*})}, \hbox{for all } k\geq 0;\quad \alpha\triangleq \frac{2\mu}{d\max\limits_{1\leq i\leq d}L_{\omega_i}}; \quad \beta \triangleq\sum_{i=1}^d\left(\frac{C_{i}^2+\nu_i^2}{\mu_{\omega_{i}}}+2L_{i}B_{i}(C_{i}+\tilde \nu_i)\right),\]
and $\g\triangleq\frac{d\max\limits_{1\leq i\leq d}L_{\omega_i}}{\mu}$. Note that from \eqref{stepsize}, we have $\g_0=\g=\frac{2}{\alpha}$. Therefore, recalling Lemma \ref{lemma:rateHarmonic}, from \eqref{ekbound2}, we have
\begin{align}\label{boundOnLy}
\EXP{\mathcal{L}(x_{k},{x^*})}\leq \frac{8\beta}{\alpha^2k}, \qquad \hbox{for all } k\geq 2.
\end{align}
From the definition of $\mathcal{L}$, the strong convexity of $\omega_i$ for all $i$, and the definition of $\|.\|$, we have 
\[\EXP{\mathcal{L}(x_{k},{x^*})}=d\sum_{i=1}^d\EXP{D_i(x_k^i,{x^*}^i)}\geq d\frac{\min\limits_{1\leq i\leq d}\mu_{\omega_i}}{2}\sum_{i=1}^d\EXP{\|x_k^i-{x^*}^i\|_i^2}=d\frac{\min\limits_{1\leq i\leq d}\mu_{\omega_i}}{2}\EXP{\|x_k-x^*\|^2}.\]
From the preceding inequality, \eqref{boundOnLy}, and the values of $\alpha$ and $\beta$ defined above, we obtain \eqref{ineq:rec-bound4}.
\end{proof}
In the next result, we consider the problem \eqref{prob:SOP} where the objective function is assumed to be convex. In this class of problems, we introduce a new averaging variant of the B-SMP algorithm and derive its convergence rate.
\begin{proposition}[Rate statement under convexity for SCOPs]\label{prop:optaveSCO} Consider problem \eqref{prob:SOP} and assume that the gradient mapping of $f$, denoted by $F$, is monotone. Let Assumptions \ref{assump:main} and \ref{assump:randvar} hold and let $\{x_k\}$ be generated by Algorithm \ref{algorithm:IRLSA-impl}. Let $r<1$ be an arbitrary scalar, the sequence $\g_k$ be non-increasing, and the sequence $\bar x_k$ be given by the following recursive rule for any $k\geq 0$:
\begin{align}
&S_{k+1}\triangleq S_k+\g_k^r,\label{def:averagingS}\\
&\bar x_{k+1}\triangleq\frac{S_k \bar x_k +\g_{k+1}^r x_{k+1}}{S_{k+1}},\label{def:averaging}
\end{align}
where we set $S_0=\g_0^r$ and $\bar x_0=x_0$. Then, for any $K\geq 0$, the following result holds:
\begin{align}\label{ineq:aveBound}
\EXP{f(\bar x_K)}-f^*\leq\left(\sum_{k=0}^{K}\g_k^r \right)^{-1}\left(2\g_{K}^{r-1}\sum_{i=1}^d p_i^{-1}L_{\omega_i}B_i^2+\sum_{i=1}^d\left(\frac{C_{i}^2+\nu_i^2}{\mu_{\omega_{i}}}+2L_{i}B_{i}(C_{i}+\tilde \nu_i)\right)\sum_{k=0}^{K}\g_k^{r+1}\right).
\end{align}
Moreover, if $P_b$ is a uniform distribution and $\g_k=\frac{\g_0}{\sqrt{k+1}}$ where $\g_0\triangleq\g \sqrt{d}$ for some $\g>0$, then for any $K> \max\{\lceil \left(\frac{3-r}{2}\right)^{\frac{2}{1-r}}\rceil,3\}$ we have
\begin{align}\label{ineq:aveRate}
\EXP{f(\bar x_K)}-f^*\leq \frac{\mathcal{B}\sqrt{d}}{\sqrt{K}},
\end{align}
where $\mathcal{B}\triangleq\left(2-r\right)2^{1-0.5r}\left(\frac{2\sum_{i=1}^dL_{\omega_i}B_i^2}{\g} +\frac{ \g\sum_{i=1}^d\left(\frac{C_{i}^2+\nu_i^2}{\mu_{\omega_{i}}}+2L_{i}B_{i}(C_{i}+\tilde \nu_i)\right)}{1-r}\right)$.
\end{proposition}
\begin{proof}
First, we use induction on $k$ to show \begin{align}\label{equ:aveFormula}\bar x_k=\sum_{t=0}^{k}\left(\frac{\g_t^r}{\sum_{j=0}^{k}\g_j^r}\right)x_t.\end{align} It holds for $k=0$, since $\bar x_0=x_0$. Let us assume \eqref{equ:aveFormula} holds for $k$. Note that from \eqref{def:averagingS}, $S_k=\sum_{j=0}^k \g_j^r$. Therefore, we have $x_k=\frac{\sum_{t=0}^{k}\g_t^rx_t}{S_k}$. From \eqref{def:averaging}, we have 
\[\bar x_{k+1}=\frac{S_k \bar x_k +\g_{k+1}^r x_{k+1}}{S_{k+1}}=\frac{\sum_{t=0}^{k}\g_t^rx_t +\g_{k+1}^r x_{k+1}}{S_{k+1}}=\frac{\sum_{t=0}^{k+1}\g_t^rx_t}{\sum_{j=0}^{k+1}\g_j^r}.\]
Thus, \eqref{equ:aveFormula} holds for any $k\geq 0$. Next, we show that \eqref{ineq:aveBound} holds. Consider relation \eqref{ineq:rec-bound} and let $x=x^*$, where $x^*$ is an arbitrary optimal solution of problem \eqref{prob:SOP}. We have
 \begin{align*}
\EXP{\mathcal{L}(x_{k+1},x^*)}&\leq \EXP{\mathcal{L}(x_k,x^*)}-\g_k\EXP{\langle F(x_k), x_k-x^*\rangle}+\sum_{i=1}^d\left(\frac{C_{i}^2+\nu_i^2}{\mu_{\omega_{i}}}+2L_{i}B_{i}(C_{i}+\tilde \nu_i)\right)\g_k^2.
\end{align*}
By the convexity of $f$, we have $\langle F(x_k), x_k-x^*\rangle\geq f(x_k)-f^*$. Therefore, we obtain
 \begin{align*}
\g_k\EXP{f(x_k)-f^*}\leq \EXP{\mathcal{L}(x_k,x^*)}-\EXP{\mathcal{L}(x_{k+1},x^*)}+\sum_{i=1}^d\left(\frac{C_{i}^2+\nu_i^2}{\mu_{\omega_{i}}}+2L_{i}B_{i}(C_{i}+\tilde \nu_i)\right)\g_k^2.
\end{align*}
Multiplying both sides by $\g_k^{r-1}$, we get
 \begin{align}\label{ineq:rec-boundn}
\g_k^r\EXP{f(x_k)-f^*}&\leq \g_k^{r-1}\EXP{\mathcal{L}(x_k,x^*)}-\g_k^{r-1}\EXP{\mathcal{L}(x_{k+1},x^*)}+\sum_{i=1}^d\left(\frac{C_{i}^2+\nu_i^2}{\mu_{\omega_{i}}}+2L_{i}B_{i}(C_{i}+\tilde \nu_i)\right)\g_k^{r+1}.\end{align}
Adding and subtracting the term $\g_{k-1}^{r-1}\EXP{\mathcal{L}(x_{k},x^*)}$, we have
 \begin{align*}
\g_k^r\EXP{f(x_k)-f^*}&\leq \g_{k-1}^{r-1}\EXP{\mathcal{L}(x_k,x^*)}-\g_{k}^{r-1}\EXP{\mathcal{L}(x_{k+1},x^*)}+\left(\g_{k}^{r-1}-\g_{k-1}^{r-1}\right)\EXP{\mathcal{L}(x_{k},x^*)}\\ &+\sum_{i=1}^d\left(\frac{C_{i}^2+\nu_i^2}{\mu_{\omega_{i}}}+2L_{i}B_{i}(C_{i}+\tilde \nu_i)\right)\g_k^{r+1}
\end{align*}
Note that using the Lipschitzian property of $\omega_i$, we get 
\fyR{\begin{align}\label{boundOnL}\mathcal{L}(x_{k+1},x^*)=\sum_{i=1}^dp_i^{-1}D_i(x_k^i,{x^*}^i)\leq \sum_{i=1}^d \frac{p_i^{-1}L_{\omega_i}}{2}\left(\|x_k^i\|_i+\|{x^*}^i\|_i\right)^2=2\sum_{i=1}^d p_i^{-1}L_{\omega_i}B_i^2.
\end{align}}
Also note that since $\g_k$ is non-increasing, and $r<1$, we have $\g_{k}^{r-1}-\g_{k-1}^{r-1}\geq 0$. Therefore, from the two preceding inequalities, we obtain 
 \begin{align*}
\g_k^r\EXP{f(x_k)-f^*}&\leq\g_{k-1}^{r-1}\EXP{\mathcal{L}(x_k,x^*)}-\g_{k}^{r-1}\EXP{\mathcal{L}(x_{k+1},x^*)}+\left(\g_{k}^{r-1}-\g_{k-1}^{r-1}\right)2\sum_{i=1}^d p_i^{-1}L_{\omega_i}B_i^2\\ &+\sum_{i=1}^d\left(\frac{C_{i}^2+\nu_i^2}{\mu_{\omega_{i}}}+2L_{i}B_{i}(C_{i}+\tilde \nu_i)\right)\g_k^{r+1}.
\end{align*}
Summing over $k$, from $k=1$ to $K$, we have  
 \begin{align*}
\sum_{k=1}^{K}\g_k^r\EXP{f(x_k)-f^*}&\leq\g_0^{r-1}\EXP{\mathcal{L}(x_1,x^*)}-\g_{K}^{r-1}\EXP{\mathcal{L}(x_{K+1},x^*)}+\left(\g_{K}^{r-1}-\g_{0}^{r-1}\right)2\sum_{i=1}^d p_i^{-1}L_{\omega_i}B_i^2\\&+\sum_{i=1}^d\left(\frac{C_{i}^2+\nu_i^2}{\mu_{\omega_{i}}}+2L_{i}B_{i}(C_{i}+\tilde \nu_i)\right)\sum_{k=1}^{K}\g_k^{r+1}.
\end{align*}
Next, we add the preceding inequality with \eqref{ineq:rec-boundn} for $k=0$:
 \begin{align*}
\sum_{k=0}^{K}\g_k^r\EXP{f(x_k)-f^*}&\leq
\left(2\sum_{i=1}^d p_i^{-1}L_{\omega_i}B_i^2\right)\g_{K}^{r-1}+\sum_{i=1}^d\left(\frac{C_{i}^2+\nu_i^2}{\mu_{\omega_{i}}}+2L_{i}B_{i}(C_{i}+\tilde \nu_i)\right)\sum_{k=0}^{K}\g_k^{r+1},
\end{align*}
where in the preceding inequality, \fyR{we employed the bound on $\EXP{\mathcal{L}(x_0,x^*)}$ given by \eqref{boundOnL}}. Dividing both sides by $\sum_{k=0}^{K}\g_k^r$, using definition of $\bar x_K$, and taking into account the convexity of $f$, we obtain the inequality \eqref{ineq:aveBound}. In the last part of the proof, we derive the rate statement given by \eqref{ineq:aveRate}. \fyR{We make use of the following inequality holding for $\g_k=\frac{\g_0}{\sqrt{k+1}}$. 
\begin{align}\label{ineq:boundForSumGammaR} 
\sum_{k=0}^K\g_k^{r+1}\leq \g_0^{(r+1)}\left(1+\frac{(K+2)^{0.5(1-r)}-1}{0.5(1-r)}\right), \quad \hbox{for all } r<1.
\end{align} The proof for this inequality is provided in Appendix \ref{A2}.}
Note that since we assumed $K>\lceil \left(\frac{3-r}{2}\right)^{\frac{2}{1-r}}\rceil$, we have 
\begin{align}\label{ineq:UBforS}
\sum_{k=0}^K\g_k^{r+1}\leq 2\g_0^{(r+1)}\left(\frac{(K+2)^{0.5(1-r)}-1}{0.5(1-r)}\right)\leq \frac{ 4\g_0^{(r+1)}(K+2)^{0.5(1-r)}}{1-r}, \quad \hbox{for all } r<1.
\end{align}
 Let us define $\theta\triangleq\sum_{i=1}^d\left(\frac{C_{i}^2+\nu_i^2}{\mu_{\omega_{i}}}+2L_{i}B_{i}(C_{i}+\tilde \nu_i)\right)$. From \eqref{ineq:UBforS}, \eqref{ineq:aveBound}, and the choice of $P_b$ being a uniform distribution, we obtain 
\begin{align*}
&\EXP{f(\bar x_K)}-f^*\leq \left(\frac{2(1-0.5r)}{\g_0^r(K+1)^{1-0.5r}}\right)\left(2d\left(\sum_{i=1}^dL_{\omega_i}B_i^2\right)\g_0^{r-1}(K+1)^{0.5(1-r)}+\frac{\theta \g_0^{r+1}(K+2)^{0.5(1-r)}}{1-r}\right)\\
&\leq \left(\frac{2-r}{\g \sqrt{d}(K+2)^{1-0.5r}}\right)\left(\frac{(K+2)^{1-0.5r}}{(K+1)^{1-0.5r}}\right)\left(2d\left(\sum_{i=1}^dL_{\omega_i}B_i^2\right) (K+2)^{0.5(1-r)}+\frac{\theta d \g^2(K+2)^{0.5(1-r)}}{1-r}\right)\\
& \leq \left(2-r\right)2^{1-0.5r}\left(\frac{2\sum_{i=1}^dL_{\omega_i}B_i^2}{\g} +\frac{\theta  \g}{1-r}\right)\frac{\sqrt{d}}{\sqrt{K+2}}.
\end{align*}
Therefore, we conclude the desired rate result.
\end{proof}
\section{Stochastic mirror-prox algorithm for SCVIs with optimal averaging}\label{sec:SMP}
In this section, our goal lies in the development of a stochastic mirror-prox algorithm to address SCVIs when the number of component sets is not huge. In contrast with the previous sections that we studied the convergence of the B-SMP algorithm, here we employ a stochastic mirror-prox algorithm in that at each iteration, all the blocks of the solution iterate are updated. Algorithm \ref{algorithm:SMP} presents the steps of the underlying method. Note that here we also employ a weighted averaging scheme similar to that of the previous section. However, the analysis of this section is different than that of Proposition \ref{prop:optaveSCO} for different reasons: (i) In this section, we do not require the Lipschitzian property of the mapping; (ii) Here we address SCVIs while Proposition \ref{prop:optaveSCO} addresses optimization problems; (iii) Our scheme here is a full-block scheme, i.e., all the blocks are updated, while the scheme in Proposition \ref{prop:optaveSCO} is a block variant of SMP method; (iv) Lastly, the averaging sequence here is $\bar y_k$, while in Proposition \ref{prop:optaveSCO} we use $\bar x_k$ as the averaging sequence. It is worth noting that in contrast with our earlier work \cite{Farzad2} where we employed a distributed stochastic approximation method for solving SCVIs, here we develop a stochastic mirror-prox method that requires two projections for each block in each iteration.

Unlike optimization problems, where the objective function provides a metric
	for measuring the performance of the algorithms, there is no immediate analog in
		variational inequality problems. Different variants of gap function have been used in the analysis of variational inequalities (cf. Chapter 10 in \cite{facchinei02finite}). To derive a convergence rate, here we use the following gap function that was also employed in \cite{Nem11}. 
\begin{definition}[Gap function]\label{def:gap1}
Let $X \subset \mathbb{R}^n$  be a non-empty and closed set. 
Suppose that mapping $F: X\rightarrow \mathbb{R}^n$ is defined on the set $X$. We define the following gap function $\hbox{G}: X \rightarrow \mathbb{R}^+\cup \{0\}$ to measure the accuracy of a vector $x \in X$: 
\begin{align}\label{equ:gapf}
G(x)= \sup_{y \in X} \langle F(y),x-y\rangle.
\end{align}
\end{definition}
It follows that $G(x)\geq 0$ for any $x \in X$ and $G(x^*)=0$ for any $x^* \in \hbox{SOL}(X,F)$ under monotonicity of $F$. We note that the function $\hbox{G}$ is indeed also a function of the set $X$ and the map $F$, but we do not \fyR{explicitly display this dependence} and {we} use $\hbox{G}$ instead of~$\hbox{G}_{X,F}$. The following result, provides the optimal convergence rate for the SMP method. \fyR{We show that the expected gap function of the \fyR{averaged} sequence $\bar y_K$ is bounded by a term of the order $\frac{1}{\sqrt{K}}$.}
\begin{proposition}[Rate statement under monotonicity for SMP algortihm for SCVIs]\label{prop:optaveSCVI} \fyR{Consider problem \eqref{eqn:SCVI}} and assume that the mapping $F$ be monotone on $X$. Let Assumption \ref{assump:main}(a,c,d) holds, and suppose $\xi$ and $\tilde \xi$ are i.i.d. random variables satisfying Assumption \ref{assump:randvar}(b,c). Let $\{x_k\}$ be generated by Algorithm \ref{algorithm:SMP}. Let $r<1$ be an arbitrary scalar, the sequence $\g_k$ be non-increasing, and the sequence $\bar y_k$ be given by \eqref{def:averagingS-SCVI}-\eqref{def:averaging-SCVI}. Then, for any $K\geq 1$ the following result holds:
\begin{align}\label{ineq:lastineqresult}
\EXP{G(\bar y_{K})}&\leq  \left(\sum_{k=0}^{K-1}\g_k^r\right)^{-1}\left(4\g_{K-1}^{r-1}\sum_{i=1}^dL_{\omega_i}B_i^2+\sum_{k=0}^{K-1}\g_k^{r+1}\sum_{i=1}^d\left(\frac{2}{\mu_{\omega_{i}}}\right)\left(2C_{i}^2 +{\tilde \nu_i}^2+1.25\nu_i^2\right)\right).
\end{align}
Moreover, if $\g_k=\frac{\g_0}{\sqrt{k+1}}$ for some $\g_0>0$, then for any $K> \max\{\lceil \left(\frac{3-r}{2}\right)^{\frac{2}{1-r}}\rceil,3\}$, we have
\begin{align*}
&\EXP{G(\bar y_{K})}\leq  \frac{\mathcal{M}}{\sqrt{K}},
\end{align*}
where $\mathcal{M}\triangleq(2-r)2^{1-0.5r}\left(\frac{4\sum_{i=1}^dL_{\omega_i}B_i^2}{\g_0}+\frac{\g_0\sum_{i=1}^d\left(\frac{2}{\mu_{\omega_{i}}}\right)\left(2C_{i}^2 +{\tilde \nu_i}^2+1.25\nu_i^2\right)}{1-r}\right)$.
\end{proposition}
\begin{algorithm}
  \caption{Stochastic mirror-prox algorithm for SCVIs}
\label{algorithm:SMP}
    \begin{algorithmic}[1]
    \STATE \textbf{initialization:} Set a random initial point $x_0\in X$, a stepsize $\g_0>0$, a scalar $r<1$, $y_0=\bar y_0=0 \in \Real^n$, and $\Gamma_0=0$;
    \FOR {$k=0,1,\ldots,K-1$}
     \FOR {$i=1,\ldots,d$}
         \STATE Update the blocks $y_k^i$  and $x_k^i$ using the following relations:
\begin{align}
y_{k+1}^{i}&:= P_{i}\left(x_{k}^{i},\g_k F_{i}(x_{k},\tilde \xi_k)\right),
\\
x_{k+1}^{i}&:=P_{i}\left(x_{k}^{i},\g_k F_{i}(y_{k+1},\xi_k)\right).
\end{align}
    \ENDFOR
    \STATE Update $\Gamma_k$ and $\bar y_{k}$ using the following recursions:
\begin{align}
&\Gamma_{k+1}:=\Gamma_k+\g_k^r,\label{def:averagingS-SCVI}\\
&\bar y_{k+1}:=\frac{\Gamma_k \bar y_k+\g_{k}^r y_{k+1}}{\Gamma_{k+1}}.\label{def:averaging-SCVI}
\end{align}
        \ENDFOR
    \STATE return $\bar y_{K};$
   \end{algorithmic}
\end{algorithm}
\begin{proof}
\fyR{The proof is done in the following three main steps: (Step 1) 
In the first step, we derive a recursive bound for a suitably defined error function. Particularly, we consider the function $\bar D:X\times \Real^n \to \Real$ defined as 
\[\bar D(x,y)\triangleq\sum_{i=1}^dD_i(x^i,y^i) \qquad \hbox{for any }x \in X \hbox{ and } y \in \Real^n.\] This function quantifies the distance of two points characterized by the block Bregman distance functions $D_i$ defined by \eqref{def:Di}; (Step 2) Given the recursive error bound in terms of $\bar D$ in Step 1, we  invoke the definition of the gap function \eqref{equ:gapf} and the averaging sequence $\bar y_{k}$ in \eqref{def:averaging-SCVI} to show the inequality \eqref{ineq:lastineqresult}; (Step 3) In the last step, under the assumption that $\g_k=\frac{\g_0}{\sqrt{k+1}}$, we use relation \eqref{ineq:lastineqresult} to derive the rate result. Below, we present the details in each step.

\noindent \textbf{(Step 1)} } For any arbitrary $i$, consider the relation $y_{k+1}^i=P_{i}\left(x_{k}^i,\g_kF_{i}(x_{k},\tilde \xi_k)\right)$. Applying Lemma \ref{lemma:proxprop}(b), we obtain
\begin{align}\label{ineq:rel1}
\g_k\langle F_{i}(x_k,\tilde \xi_k), y_{k+1}^{i}-x_{k+1}^{i}\rangle +D_i(x_{k}^{i},y_{k+1}^{i}) +D_{i}(y_{k+1}^{i},x_{k+1}^{i})\leq 
D_{i}(x_{k}^{i},x_{k+1}^{i}).
\end{align}
Let vector $x \in X$ be given. Similarly, from $x_{k+1}^i=P_{i}\left(x_{k}^i,\g_kF_{i}(y_{k+1},\xi_{k})\right)$ and Lemma \ref{lemma:proxprop}(b), we obtain 
\begin{align*}
\g_k\langle \fyR{F_{i}(y_{k+1},\xi_k)}, x_{k+1}^{i}-x^{i}\rangle +D_{i}(x_{k}^{i},x_{k+1}^{i}) +D_{i}(x_{k+1}^{i},x^{i})\leq 
D_{i}(x_{k}^{i},x^{i}).
\end{align*}
Adding and subtracting $y_{k+1}^{i}$, the preceding relation yields
\begin{align}\label{ineq:rel2}
\g_k\langle \fyR{F_{i}(y_{k+1},\xi_k)}, x_{k+1}^{i}-y_{k+1}^{i}\rangle +\g_k\langle F_{i}(y_{k+1},\xi), y_{k+1}^{i}-x^{i}\rangle+D_{i}(x_{k}^{i},x_{k+1}^{i}) +D_{i}(x_{k+1}^{i},x^{i})\leq 
D_{i}(x_{k}^{i},x^{i}).
\end{align}
Adding \eqref{ineq:rel1} and \eqref{ineq:rel2}, we obtain
\begin{align*}
&\g_k\langle \fyR{F_{i}(y_{k+1},\xi_k)}-F_{i}(x_k,\tilde \xi_k), x_{k+1}^{i}-y_{k+1}^{i}\rangle +\g_k\langle \fyR{F_{i}(y_{k+1},\xi_k)}, y_{k+1}^{i}-x^{i}\rangle  \notag\\ &+D_{i}(x_{k}^{i},y_{k+1}^{i}) +D_{i}(y_{k+1}^{i},x_{k+1}^{i})+D_{i}(x_{k+1}^{i},x^{i})\leq 
D_{i}(x_{k}^{i},x^{i}) .
\end{align*}
Rearranging the terms in the preceding inequality, we obtain
\begin{align*}
D_i(x_{k+1}^i,x^i) &\leq D_i(x_{k}^i,{x}^i)+\g_k\langle \fyR{F_{i}(x_k,\tilde \xi_k)-F_{i}(y_{k+1},\xi_k)}, x_{k+1}^{i}-y_{k+1}^{i}\rangle\\
& +\g_k\langle F_{i}(y_{k+1},\xi), {x}^{i}-y_{k+1}^{i}\rangle -D_{i}(x_{k}^{i},y_{k+1}^{i})-D_{i}(y_{k+1}^{i},x_{k+1}^{i})
\end{align*}
Using the definition of stochastic errors $\fyR{\tilde w_k}$ and $w_k$, in the preceding result, we can substitute $F_{i}(x_k,\tilde \xi_k)$ by $F_i(x_k)+{\tilde w_k}^i$, and $F_{i}(y_{k+1},\xi_k)$ by $F_i(y_{k+1})+{w_k}^i$. We then obtain
\begin{align*}
D_i(x_{k+1}^i,x^i) &\leq D_i(x_{k}^i,{x}^i)+\g_k\underbrace{\langle F_{i}(x_k)-F_{i}(y_{k+1}), x_{k+1}^{i}-y_{k+1}^{i}\rangle}_{\mbox{Term}\ 1}\\
& +\g_k\langle F_{i}(y_{k+1}), {x}^{i}-y_{k+1}^{i}\rangle -D_{i}(x_{k}^{i},y_{k+1}^{i})- D_{i}(y_{k+1}^{i},x_{k+1}^{i})\\
& +\g_k\underbrace{\langle {\fyR{\tilde w_k}}^{i}-w_k^{i}, x_{k+1}^{i}-y_{k+1}^{i}\rangle}_{\mbox{Term}\ 2}+\g_k\langle w_k^{i}, {x}^{i}-y_{k+1}^{i}\rangle.
\end{align*}
Applying Fenchel's inequality to Term $1$, we obtain 
\begin{align}\label{term1}\mbox{Term}\ 1&=\left\langle \sqrt{\frac{2}{\mu_{\omega_{i}}}}\fyR{\left(F_{i}(x_k)-F_{i}(y_{k+1})\right)}, \sqrt{\frac{\mu_{\omega_{i}}}{2}}\left(x_{k+1}^{i}-y_{k+1}^{i}\right)\right\rangle \notag\\ 
&\leq  \frac{1}{2}\left(\frac{2}{\mu_{\omega_i}}\right)\| F_{i}(x_k)-F_{i}(y_{k+1})\|_{*i}^2+\frac{1}{2}\left(\frac{\mu_{\omega_i}}{2}\right)\|x_{k+1}^{i}-y_{k+1}^{i}\|_{i}^2.\end{align}
Similarly, we may obtain a bound on $\mbox{Term}\ 2$. Using strong convexity of $\omega_{i}$ (cf. Lemma \ref{lemma:proxprop}(a)), we conclude
\begin{align*}
D_i(x_{k+1}^i,x^i)&\leq D_i(x_{k}^i,{x}^i)+\g_k^2\left(\frac{1}{\mu_{\omega_{i}}}\right)\| F_{i}(x_k)-F_{i}(y_{k+1})\|_{*i}^2+ \left(\frac{\mu_{\omega_i}}{4}\right)\|x_{k+1}^{i}-y_{k+1}^{i}\|_{i}^2\\
& +\g_k\langle F_{i}(y_{k+1}), {x}^{i}-y_{k+1}^{i}\rangle -\frac{\mu_{\omega_i}}{2}\|x_{k}^{i}-y_{k+1}^{i}\|_{i}^2- \frac{\mu_{\omega_i}}{2}\|y_{k+1}^{i}-x_{k+1}^{i}\|_{i}^2\\
& +\g_k^2\left(\frac{1}{\mu_{\omega_{i}}}\right)\| {\tilde w}_k^{i}-w_k^{i}\|_{*{i}}^2+\left(\frac{\mu_{\omega_{i}}}{4}\right)\|x_{k+1}^{i}-y_{k+1}^{i}\|_{i}^2+\g_k\langle w_k^{i}, {x}^{i}-y_{k+1}^{i}\rangle.
\end{align*}
Invoking Assumption \ref{assump:main}(c) and that for any $a,b \in \Real$, $(a+b)^2\leq 2(a^2+b^2)$, we have 
\begin{align*}
D_i(x_{k+1}^i,x^i)&\leq D_i(x_{k}^i,{x}^i)+\left(\frac{4}{\mu_{\omega_{i}}}\right)\g_k^2C_i^2+\g_k\langle F_{i}(y_{k+1}), {x}^{i}-y_{k+1}^{i}\rangle -\frac{\mu_{\omega_i}}{2}\|x_{k}^{i}-y_{k+1}^{i}\|_{i}^2\\
& +\g_k^2\left(\frac{1}{\mu_{\omega_i}}\right)\| {\fyR{\tilde w_k}}^{i}-w_k^{i}\|_{*{i}}^2+\g_k\langle w_k^{i}, {x}^{i}-y_{k+1}^{i}\rangle \\
& \leq D_i(x_{k}^i,{x}^i)+\left(\frac{4}{\mu_{\omega_i}}\right)\g_k^2C_{i}^2+\g_k\langle F_{i}(y_{k+1}), {x}^{i}-y_{k+1}^{i}\rangle \\
&+\g_k^2\left(\frac{1}{\mu_{\omega_{i}}}\right)\| {\fyR{\tilde w_k}}^{i}-w_k^{i}\|_{*{i}}^2+\g_k\langle w_k^{i}, {x}^{i}-y_{k+1}^{i}\rangle.
\end{align*}
Using the triangle inequality for the dual norm $\|\cdot\|_{*i}$, from the preceding inequality we have  
\begin{align}\label{ineq:intermediate0}
D_i(x_{k+1}^i,x^i)&\leq D_i(x_{k}^i,{x}^i) +\left(\frac{4}{\mu_{\omega_{i}}}\right)\g_k^2C_{i}^2 +\g_k^2\left(\frac{2}{\mu_{\omega_{i}}}\right)\| {\fyR{\tilde w_k}}^{i}\|_{*{i}}^2+\g_k^2\left(\frac{2}{\mu_{\omega_{i}}}\right)\|w_k^{i}\|_{*{i}}^2\notag\\ &+\g_k\langle F_{i}(y_{k+1}), {x}^{i}-y_{k+1}^{i}\rangle+\g_k\langle w_k^{i}, {x}^{i}-y_{k+1}^{i}\rangle.
\end{align}
We now estimate the term $\g_k\langle w_k^{i}, {x}^{i}-y_{k+1}^{i}\rangle$. Let us define a sequence of vectors $u_{k}\triangleq(u_k^1;u_k^2;\ldots;u_k^d)$ as follows:
\begin{align}\label{def:ut}
u_{k+1}^{i}=P_{i}(u_k^{i},-\g_kw_{k}^{i}), \quad \hbox{for all } k\geq 0,
\end{align}
where $u_0=x_0$. We can write 
\begin{align}\label{boundOnLastTerm}\g_k\langle w_k^{i}, {x}^{i}-y_{k+1}^{i}\rangle =\langle \g_kw_k^{i}, {x}^{i}-u_{k}^{i}\rangle+\g_k\langle w_k^{i}, {u}_k^{i}-y_{k+1}^{i}\rangle. \end{align}
Applying Lemma \ref{lemma:proxprop}(c) and taking to account the definition of $u_k$ in \eqref{def:ut}, we have
\begin{align*}
 D_{i}(u_{k+1}^{i},x^{i})\leq D_{i}(u_k^{i},x^{i})-\langle \g_kw_k^{i}, {x}^{i}-u_{k}^{i}\rangle+\frac{\g_k^2}{2\mu_{\omega_{i}}}\|w_k^{i}\|_{*i}^2.
\end{align*}
Therefore, from the preceding relation and \eqref{boundOnLastTerm}, we obtain
\begin{align*}
 \g_k\langle w_k^{i}, x^{i}-y_{k+1}^{i}\rangle  \leq D_i(u_{k}^i,x^i)-D_i(u_{k+1}^i,x^i)+\frac{\g_k^2}{2\mu_{\omega_{i}}}\|w_k^{i}\|_{*i}^2+\g_k\langle w_k^{i}, {u}_k^{i}-y_{k+1}^{i}\rangle.
\end{align*} 
From \eqref{ineq:intermediate0} and the preceding relation, we obtain 
\begin{align*}
D_i(x_{k+1}^i,x^i)&\leq D_i(x_{k}^i,{x}^i) +\left(\frac{2\g_k^2}{\mu_{\omega_{i}}}\right)\left(2C_{i}^2 +\| {\fyR{\tilde w_k}}^{i}\|_{*{i}}^2+1.25\|w_k^{i}\|_{*{i}}^2\right)\\ 
&+D_i(u_{k}^i,x^i)-D_i(u_{k+1}^i,x^i)+\g_k\langle F_{i}(y_{k+1}), {x}^{i}-y_{k+1}^{i}\rangle+\g_k\langle w_k^{i}, {u}_k^{i}-y_{k+1}^{i}\rangle.
\end{align*}
By summing both sides of the preceding relation over $i$, invoking the definition of function $\bar D$, and the aggregated inner product, we obtain 
\begin{align*}
\bar D(x_{k+1},x)&\leq \bar D(x_{k},{x}) +\sum_{i=1}^d\left(\frac{2\g_k^2}{\mu_{\omega_{i}}}\right)\left(2C_{i}^2 +\| {\fyR{\tilde w_k}}^{i}\|_{*{i}}^2+1.25\|w_k^{i}\|_{*{i}}^2\right)\\ 
&+\bar D(u_{k},x)-\bar D(u_{k+1},x)+\g_k\langle F(y_{k+1}), {x}-y_{k+1}\rangle+\g_k\langle w_k, {u}_k-y_{k+1}\rangle.
\end{align*}

\fyR{\noindent \textbf{(Step 2)}} Next, using monotonicity of $F$ and by rearranging the terms, we further obtain
\begin{align}\label{ineq:boundsomething1}
\g_k\langle F(x), y_{k+1}-x\rangle&\leq  \left(\bar D(x_{k},{x})+\bar D(u_{k},x)\right)-\left(\bar D(x_{k+1},x)+\bar D(u_{k+1},x)\right)\notag\\ 
&+\g_k^2\sum_{i=1}^d\left(\frac{2}{\mu_{\omega_{i}}}\right)\left(2C_{i}^2 +\| {\fyR{\tilde w_k}}^{i}\|_{*{i}}^2+1.25\|w_k^{i}\|_{*{i}}^2\right)+\g_k\langle w_k, {u}_k-y_{k+1}\rangle.
\end{align}
Next, multiplying both sides by $\g_k^{r-1}$, and adding and subtracting $\g_{k-1}^{r-1}\left(\bar D(x_{k},{x})+\bar D(u_{k},x)\right)$, we obtain 
\begin{align*}
\g_k^r\langle F(x), y_{k+1}-x\rangle&\leq  \g_{k-1}^{r-1}\left(\bar D(x_{k},{x})+\bar D(u_{k},x)\right)-\g_k^{r-1}\left(\bar D(x_{k+1},x)+\bar D(u_{k+1},x)\right)\\ 
& +\left(\g_{k}^{r-1}-\g_{k-1}^{r-1}\right)\left(\bar D(x_{k},{x})+\bar D(u_{k},x)\right)\\
&+\g_k^{r+1}\sum_{i=1}^d\left(\frac{2}{\mu_{\omega_{i}}}\right)\left(2C_{i}^2 +\| {\fyR{\tilde w_k}}^{i}\|_{*{i}}^2+1.25\|w_k^{i}\|_{*{i}}^2\right)+\g_k^r\langle w_k, {u}_k-y_{k+1}\rangle.
\end{align*}
Note that since $\g_k$ is non-increasing and $r<1$, we have $\g_{k}^{r-1}-\g_{k-1}^{r-1}\geq 0$. Note furthur that using Lemma \ref{lemma:proxprop}(a) and the definition of $\bar D$, we have\begin{align}\label{ieq:boundonDbar}\bar D(x_{k},{x}) \leq \sum_{i=1}^d\frac{L_{\omega_i}}{2}\|x_k^i-x^i\|_i^2\leq \sum_{i=1}^d\frac{L_{\omega_i}}{2}(\|x_k^i\|_i+\|x^i\|_i)^2\leq 2\sum_{i=1}^dL_{\omega_i}2C_i^2.\end{align}
Similarly, we get $\bar D(u_{k},{x})\leq 2\sum_{i=1}^dL_{\omega_i}B_i^2$.
Using these bounds, summing over $k$ from $1$ to $K-1$, and then removing the negative terms on the right-hand side of the resulting inequality, we obtain
\begin{align*}
\sum_{k=1}^{K-1}\langle F(x), \g_k^ry_{k+1}-x\rangle&\leq  \g_{0}^{r-1}\left(\bar D(x_{1},{x})+\bar D(u_{1},x)\right)+\left(\g_{K-1}^{r-1}-\g_{0}^{r-1}\right)4\sum_{i=1}^dL_{\omega_i}B_i^2\\
&+\sum_{k=1}^{K-1}\g_k^{r+1}\sum_{i=1}^d\left(\frac{2}{\mu_{\omega_{i}}}\right)\left(2C_{i}^2 +\| {\fyR{\tilde w_k}}^{i}\|_{*{i}}^2+1.25\|w_k^{i}\|_{*{i}}^2\right)+\sum_{k=1}^{K-1}\g_k^r\langle w_k, {u}_k-y_{k+1}\rangle.
\end{align*}
Consider \eqref{ineq:boundsomething1} for $k=0$. By multiplying both sides of that relation by $\g_0^{r-1}$, and then summing with the preceding inequality, we obtain
\begin{align*}
\langle F(x), \sum_{k=0}^{K-1}\g_k^r y_{k+1}-x\rangle&\leq  \g_{0}^{r-1}\left(\bar D(x_{0},{x})+\bar D(u_{0},x)\right)+\left(\g_{K-1}^{r-1}-\g_{0}^{r-1}\right)4\sum_{i=1}^dL_{\omega_i}B_i^2\\
&+\sum_{k=0}^{K-1}\g_k^{r+1}\sum_{i=1}^d\left(\frac{2}{\mu_{\omega_{i}}}\right)\left(2C_{i}^2 +\| {\fyR{\tilde w_k}}^{i}\|_{*{i}}^2+1.25\|w_k^{i}\|_{*{i}}^2\right)+\sum_{k=0}^{K-1}\g_k^r\langle w_k, {u}_k-y_{k+1}\rangle.
\end{align*}
Dividing both sides by $\sum_{k=0}^{K-1}\g_k^r$, invoking the definition of $\bar y_{K}$, and using \eqref{ieq:boundonDbar}, we have 
\begin{align*}
\langle F(x), \bar y_{K}-x\rangle&\leq  \left(\sum_{k=0}^{K-1}\g_k^r\right)^{-1}\left(4\g_{K-1}^{r-1}\sum_{i=1}^dL_{\omega_i}B_i^2+\sum_{k=0}^{K-1}\g_k^{r+1}\sum_{i=1}^d\left(\frac{2}{\mu_{\omega_{i}}}\right)\left(2C_{i}^2 +\| {\fyR{\tilde w_k}}^{i}\|_{*{i}}^2+1.25\|w_k^{i}\|_{*{i}}^2\right)\right.\\
&\left.+\sum_{k=0}^{K-1}\g_k^r\langle w_k, {u}_k-y_{k+1}\rangle\right).
\end{align*}
Note that since the right-hand side of the preceding relation is independent of $x$, taking supremum from the left-hand side, and from definition of the G function, we obtain \begin{align*}
G(\bar y_{K})&\leq  \left(\sum_{k=0}^{K-1}\g_k^r\right)^{-1}\left(4\g_{K-1}^{r-1}\sum_{i=1}^dL_{\omega_i}B_i^2+\sum_{k=0}^{K-1}\g_k^{r+1}\sum_{i=1}^d\left(\frac{2}{\mu_{\omega_{i}}}\right)\left(2C_{i}^2 +\| {\fyR{\tilde w_k}}^{i}\|_{*{i}}^2+1.25\|w_k^{i}\|_{*{i}}^2\right)\right.\\
&\left.+\sum_{k=0}^{K-1}\g_k^r\langle w_k, {u}_k-y_{k+1}\rangle\right).
\end{align*}

\fyR{\noindent \textbf{(Step 3)}} Next, taking expectation on both sides, and using Assumption \ref{assump:randvar}(b,c), we can obtain the inequality \eqref{ineq:lastineqresult}.
In the remainder of the proof, we derive the rate statement. For $\g_k=\frac{\g_0}{\sqrt{k+1}}$, in a similar fashion to the proof of \fyR{ inequality \eqref{ineq:boundForSumGammaR}}, for $K\geq 4$, we have 
\begin{align}\label{ineq:LBforSII}
\sum_{k=0}^{K-1}\g_k^{r}\geq \frac{\g_0^{r}(K+1)^{1-0.5r}}{2(1-0.5r)}, \quad \hbox{for all } r<1.
\end{align}
Also, since we assumed that $K>\lceil \left(\frac{3-r}{2}\right)^{\frac{2}{1-r}}\rceil$, we see that 
\begin{align}\label{ineq:UBforSII}
\sum_{k=0}^{K-1}\g_k^{r+1}\leq \frac{ 4\g_0^{(r+1)}(K+1)^{0.5(1-r)}}{1-r}, \quad \hbox{for all } r<1.
\end{align}
From \eqref{ineq:LBforSII}, \eqref{ineq:UBforSII}, and \eqref{ineq:lastineqresult}, we obtain 
\begin{align*}
&\EXP{G(\bar y_{K})}\leq \left(\frac{2(1-0.5r)}{\g_0^rK^{1-0.5r}}\right)\left(4\g_0^{r-1}K^{0.5(1-r)}\sum_{i=1}^dL_{\omega_i}B_i^2+\frac{\mathcal{C}\g_0^{r+1}(K+1)^{0.5(1-r)}}{1-r}\right)\\
&\leq \left(\frac{2-r}{\g_0(K+1)^{1-0.5r}}\right)\left(\frac{(K+1)^{1-0.5r}}{K^{1-0.5r}}\right)\left(4\sum_{i=1}^dL_{\omega_i}B_i^2(K+1)^{0.5(1-r)}+\frac{\mathcal{C}\g_0^2(K+1)^{0.5(1-r)}}{1-r}\right)\\
& \leq (2-r)2^{1-0.5r}\left(\frac{4\sum_{i=1}^dL_{\omega_i}B_i^2}{\g_0}+\frac{\mathcal{C}\g_0}{1-r}\right)\frac{1}{\sqrt{K+1}}.
\end{align*}
where $\mathcal{C}\triangleq\sum_{i=1}^d\left(\frac{2}{\mu_{\omega_{i}}}\right)\left(2C_{i}^2 +{\tilde \nu_i}^2+1.25\nu_i^2\right)$. Therefore, we obtain the desired rate statement.
\end{proof}
\section{Concluding remarks}\label{sec:conRem}
Motivated by the challenges arising from computation of \fyR{solution} to the large-scale stochastic Cartesian variational inequality problems, in the first part of the paper, we develop a randomized block coordinate stochastic mirror-prox (B-SMP) algorithm. At each iteration, only a randomly selected block coordinate of the solution iterate is updated through implementing two consecutive projection steps. Under the assumption of strict pseudo-monotonicity of the mapping, first we prove convergence of the generated sequence to the solution set of the problem in an almost sure sense. To derive the  rate statements, we consider SCVI problems with strongly pseudo-monotone mappings and derive the optimal convergence rate in terms of the problem parameters, prox mapping parameters, iteration number, and number of blocks. We then consider stochastic convex optimization problems on sets with block structures. Under a new weighted averaging scheme, we derive the associated convergence rate. In the second part of the paper, we develop a stochastic mirror-prox algorithm for solving SCVIs where all the blocks are updated at each iteration. We show that using a different weighted averaging sequence, the optimal convergence rate can be achieved.

\section{Appendix}
\subsection{Proof of Lemma \ref{lem:uniqueness}}\label{A1}
We show that VI$(X,F)$ has a unique solution. To arrive at a contradiction, let us assume $x^*, {x^*}' \in $ SOL$(X,F)$ such that $x^*\neq {x^*}'$. Since $x^*$ is a solution to VI$(X,F)$ and that 
${x^*}' \in X$, we have \begin{align}\label{ineq:spse1}\langle F(x^*),{x^*}'-x^*\rangle \ge 0.\end{align} Similarly, since ${x^*}'$ is a solution to VI$(X,F)$ and that ${x^*} \in X$, we have 
\begin{align}\label{ineq:spse2}\langle F({x^*}'),x^*-{x^*}'\rangle \ge 0.\end{align}
From the preceding two inequalities, we obtain 
\begin{align}\label{ineq:spse3}\langle F(x^*)-F({x^*}'),x^*-{x^*}'\rangle \le 0.\end{align}
Next, invoking the definition of strict pseudo-monotonicity of $F$, from \eqref{ineq:spse1}, we obtain $\langle F({x^*}'),{x^*}'-x^*\rangle > 0.$ Simiraly, from \eqref{ineq:spse2}, we have $\langle F({x^*}),x^*-{x^*}'\rangle > 0$. Summing the preceding two inequalities, we have \begin{align*}\langle F({x^*})-F({x^*}'),x^*-{x^*}'\rangle > 0,\end{align*}
which contradicts \eqref{ineq:spse3}. \fyR{Therefore, $x^*={x^*}'$ implying that VI$(X,F)$ has at most one solution. From Remark \ref{rem1}, we conclude that VI$(X,F)$ has a unique solution.}
\subsection{Proof of inequality \eqref{ineq:boundForSumGammaR}}\label{A2}
\fyR{To show this inequality, consider the following two cases:\\ (i) If $r\leq 0$, we have
\[\g_0^{-r}\sum_{k=0}^K\g_k^{r}=\sum_{k=1}^{K+1}k^{-\frac{r}{2}}\geq \int_{0}^{K+1}x^{-0.5r}dx\geq \int_{1}^{K+1}x^{-0.5r}dx=\frac{(K+1)^{1-0.5r}-1}{1-0.5r}.\]
(ii) If $0<r<1$, we have 
\[\g_0^{-r}\sum_{k=0}^K\g_k^{r}=\sum_{k=1}^{K+1}k^{-\frac{r}{2}}\leq \int_{1}^{K+2}x^{-0.5r}dx\geq \int_{1}^{K+1}x^{-0.5r}dx=\frac{(K+1)^{1-0.5r}-1}{1-0.5r}.\]}
Therefore, from the two preceding relations, we have 
\begin{align*}
\sum_{k=0}^K\g_k^{r}\geq \g_0^{r}\left(\frac{(K+1)^{1-0.5r}-1}{1-0.5r}\right), \quad \hbox{for all } r<1.
\end{align*}
\subsection{Proof of Lemma \ref{lemma:rateHarmonic}}\label{A3}
We use induction to show \eqref{ekbound1}. For $k=K$, we have
\[e_K=\frac{Ke_K}{K}\leq \frac{\max \{\frac{\beta\g^2}{\alpha \g-1},Ke_K\}}{K},\]
implying that \eqref{ekbound1} holds for $k=K$. Let us assume \eqref{ekbound1} holds for some $k\geq K$. Note that we have 
\[k\geq \lceil \alpha \g\rceil \geq \alpha \g \Rightarrow 1-\frac{\alpha \g}{k}\geq0.\]Then, using the preceding inequality, relation \eqref{ekbound0}, $\g_k=\frac{\g}{k}$, and the induction hypothesis, we can write
\begin{align}\label{ekbound3}e_{k+1}\leq \left(1-\frac{\alpha\g}{k}\right)e_k+\frac{\beta \g^2}{k^2}\leq \left(1-\frac{\alpha\g}{k}\right)\frac{\theta}{k}+\frac{\beta \g^2}{k^2},\end{align}
where $\theta\triangleq\max \{\frac{\beta\g^2}{\alpha \g-1},Ke_K\}$. From the definition of $\theta$, we have
\begin{align*}
&\theta\geq \frac{\beta\g^2}{\alpha \g-1} \Rightarrow \theta \alpha \g -\beta\g^2 \geq \theta \Rightarrow \frac{\alpha \g \theta -\beta \g^2}{\theta}\geq 1>\frac{k}{k+1}\Rightarrow \frac{\alpha \g \theta -\beta \g^2}{k}\geq \frac{\theta}{k+1}\\
& \Rightarrow \frac{\alpha \g \theta -\beta \g^2}{k^2}\geq \frac{\theta}{k(k+1)}\Rightarrow \frac{\alpha \g \theta}{k^2} -\frac{\beta \g^2}{k^2}\geq \frac{\theta}{k}-\frac{\theta}{k+1}\Rightarrow \frac{\theta}{k+1}\geq \left(1-\frac{\alpha\g}{k}\right)\frac{\theta}{k}+\frac{\beta \g^2}{k^2}.
\end{align*}
Invoking \eqref{ekbound3}, we have $e_{k+1}\leq \frac{\theta}{k+1}$. Therefore, \eqref{ekbound1} holds for all $k\geq K$. Next we show \eqref{ekbound2}. Note that since $1< \alpha \g = 2$, we have $K=\lceil \alpha \g\rceil=2$. Therefore, from \eqref{ekbound2} and $K=2$, we have 
\begin{align*}
e_k\leq \frac{\max \{\frac{\beta\g^2}{\alpha \g-1},2e_2\}}{k}, \qquad \hbox{for all } k\geq 2.
\end{align*}
Replacing $\g$ by $\frac{2}{\alpha}$, we obtain \begin{align}\label{ekbound4}
e_k\leq \frac{\max \{\frac{4\beta}{\alpha^2},2e_2\}}{k}, \qquad \hbox{for all } k\geq 2.
\end{align}
Next we find a bound on $e_2$. From \eqref{ekbound0}, $\g=\frac{2}{\alpha}$, and that $\g_1=\g$, we have \[e_2\leq (1-\alpha \g)e_1+\beta \g^2=\frac{4\beta}{\alpha^2}-e_1 \leq \frac{4\beta}{\alpha^2},\]
where the last inequality is due to non-negativity of $e_1$. Therefore, using the preceding relation and \eqref{ekbound4}, we obtain the desired result. 

\bibliographystyle{abbrv}
\bibliography{SVAA}
\end{document}